\documentclass{amsart}
\usepackage{amsmath,amssymb}
\usepackage{amsfonts}
\usepackage{color}
\usepackage[left=3.5cm,right=3.5cm, top=4cm,bottom=4cm]{geometry}

\newcommand{  \red  }{\color{black}}
\newcommand{ \black }{\color{black} }

\newcommand{\Red}{\color{black}}

\usepackage[english]{babel}

\newtheorem{thm}{Theorem}{\bf}{\it}
\newtheorem{lem}[thm]{Lemma}{\bf}{\it}
{\bf}{\it}
{\bf}{\it}
{\bf}{\it}
{\bf}{\it}
{\bf}{\it}
{\bf}{\it}
{\bf}{\it}
{\bf}{\it}
{\bf}{\it}

\theoremstyle{definition}
\newtheorem{defn}[thm]{Definition}{\bf}{\rm}
{\bf}{\rm}
{\bf}{\rm}
\newtheorem{rem}[thm]{Remark}{\bf}{\rm}

\newcommand{\nab}{\langle \nabla \rangle_c}
\newcommand{\nabo}{\langle \nabla \rangle_0}

\newcommand {\aplt} {\ {\raise-.5ex\hbox{$\buildrel<\over\sim$}}\ }

\newcommand{\e}{\mathrm{e}}
\newcommand{\Ac}{\mathcal{A}_c}

\newcommand{\dd}{\mathrm{d}}

\newcommand{\coso}{\mathrm{cos}\left(\tau \langle \nabla \rangle_0\right)}
\newcommand{\sino}{\mathrm{sin}\left(\tau \langle \nabla \rangle_0\right)}
\newcommand{\sinco}{\mathrm{sinc}\left(\tau \langle \nabla \rangle_0\right)}

\newcommand{\ua}{u^\ast}

\renewcommand{\u}{u_\ell}
\newcommand{\F}{F_\ell}
\newcommand{\n}{n_\ell}
\newcommand{\np}{\dot{n}_\ell}
\renewcommand{\S}{S^F_\ell}
\newcommand{\tl}{(t_\ell)}

\newcommand{\B}{B_{t_\ell,r}}

\usepackage{tikz}
\usepackage{pgfplots}

\begin{document}

\author{Simon Baumstark}
\address{Fakult\"{a}t f\"{u}r Mathematik, Karlsruhe Institute of Technology,
Englerstr. 2, 76131 Karlsruhe, Germany}
\email{simon.baumstark@kit.edu}

\author{Katharina Schratz}
\address{Fakult\"{a}t f\"{u}r Mathematik, Karlsruhe Institute of Technology,
Englerstr. 2, 76131 Karlsruhe, Germany}
\email{katharina.schratz@kit.edu}

\title[Uniformly accurate oscillatory  integrators for the KGZ system]{Uniformly accurate   oscillatory  integrators for the Klein-Gordon-Zakharov system from low- to high-plasma frequency regimes}

\begin{abstract}
We present a novel class of oscillatory  integrators for the Klein-Gordon-Zakharov  system which are  uniformly accurate with respect to the  plasma frequency $c$. Convergence holds from the slowly-varying low-plasma  up to the highly oscillatory high-plasma frequency regimes without any step size restriction and, especially, uniformly in $c$. The introduced schemes are moreover asymptotic consistent and approximates the solutions of the corresponding Zakharov limit system in the high-plasma frequency limit ($c \to \infty$).  \Red We in particular present the construction of the first- and second-order uniformly accurate oscillatory integrators and establish rigorous, uniform error estimates. Numerical experiments underline our theoretical convergence results.\black
\end{abstract}

\keywords{}

\maketitle

\section{Introduction}
The Klein-Gordon-Zakharov (KGZ) system
\begin{equation}\label{eq:kgzOr}
\begin{aligned}
 c^{-2}\partial_{tt}z - \Delta z +c^2 z &=  - n z,\qquad\,\,\,  z(0) = z_0,\, \partial_t z(0) = c^2z'_0,\\
 \partial_{tt}n-\Delta n & = \Delta \vert z\vert^2,\qquad  n(0) = n_0, \, \partial_t n(0) = \dot{n}_0
\end{aligned}
\end{equation}
describes the interaction of Langmuir waves (oscillations of the electron density) with ion sound waves in a plasma. For existence and uniqueness of (global) smooth solutions and physical applications of the system we refer to \cite{Bell06,LWP1a,LWP1b,LWP2} and the references therein.

The so-called  \textit{low-plasma frequency regime} $c=1$ of the Klein-Gordon-Zakharov system  is nowadays well understood and extensively studied numerically, see, e.g.,   \cite{WCZ} \Red for an energy conservative finite difference method \black and \cite{BaoKGZ1,Zhao16} for exponential wave integrator methods.

From an analytical point of view, also the \emph{high-plasma frequency regime} $ c \gg 1$ has  gained a lot of attention lately and is, meanwhile, in large parts well understood mathematically. In particular, the high-plasma frequency limit, i.e., $c \to \infty$, where the Klein-Gordon-Zakharov system \eqref{eq:kgzOr} converges to the Zakharov system 
\begin{equation}\label{zak}
\begin{aligned}
2  i \partial_t z\textstyle - \Delta z & =  - n z, \\
  \partial_{tt}n-\Delta n & =  \frac12 \Delta \vert z \vert^2
\end{aligned}
\end{equation}
is nowadays extensively studied mathematically, see, e.g., \cite{Berg96,LWP1a,LWP1b,Texier07} and the references therein.

In contrast, only very little is known in the numerical analysis of the KGZ system \eqref{eq:kgzOr}  in \textit{weakly to strongly high-plasma frequency regimes} $c>1$ up to $c \gg 1$. The main challenge in these regimes lies in the fact that the solution becomes highly oscillatory in time. Classical numerical schemes fail to resolve these oscillations correctly, and, in particular, severe step size restrictions need to be imposed to allow a sufficient close approximation to exact solutions. This, however, leads to huge computational efforts and non optimal convergence. For an introduction to the numerical analysis of highly oscillatory problems we refer to \cite{EFHI09,HLW} and in the context of the classical Klein-Gordon equation in particular to \cite{BG,FS13}, as well as the  references therein.

In recent years, many achievements have been accomplished in the numerical analysis of the classical Klein-Gordon equation (i.e., $n = - \vert z \vert^2$ in KGZ)
\begin{equation}\label{kgr}
c^{-2}\partial_{tt} z - \Delta z + c^2 z = |z|^2 z
\end{equation}
in the highly oscillatory so-called non-relativistic regime $c \gg 1$. In particular, uniformly accurate schemes for the classical Klein-Gordon equation~\eqref{kgr}  were derived in \cite{BaoZ,ChC,BFS17} based on mutliscale and modulated Fourier expansion techniques. For the latter, see also  \cite{CHL08a,CoHaLu03,HL,HLW} and the references therein. The great benefit of \emph{uniformly accurate numerical schemes} lies  in their optimal approximation property: They allow convergence independently of the parameter $c$, i.e., from classical slowly-varying $c=1$ up to  highly oscillatory $c \gg 1$ regimes.  Nevertheless, note that the Klein-Gordon equation~\eqref{kgr} benefits from a much simpler mathematical structure than the KGZ system \eqref{eq:kgzOr} which makes it, compared to the KGZ system, a much easier task to construct uniformly accurate methods and establish the corresponding rigorous error estimates.

A uniformly accurate multiscale time integrator spectral method for  the KGZ system  \eqref{eq:kgzOr} was recently introduced in \cite{BaoKGZUA} and its convergence was extensively studied numerically. However, a rigorous convergence analysis is still lacking and, up to our knowledge, no rigorous error analysis  for numerical schemes for the KGZ system~\eqref{eq:kgzOr} in weakly to strongly high-plasma frequency regimes is known so far due to the difficult coupling of the problem.

In this work we introduce a novel concept of \emph{uniformly accurate oscillatory  integrators for the KGZ system}~\eqref{eq:kgzOr} which converge uniformly in the parameter $c$. The new techniques developed in this work in particular allow us  to establish rigorous  error estimates not only in low-plasma, but for the first time also  in weakly to strongly high-plasma frequency regimes.  Our idea is thereby inspired by the recent work \cite{HerrS17} which allowed us to \emph{numerically overcome the loss of derivative}  within the Zakharov limit system \eqref{zak}  as well as  techniques related to  \emph{twisting the variable} in order to precisely capture the oscillations within the solution. The latter was introduced for the classical Klein-Gordon equation in \cite{BFS17} (see also \cite{HS16,OS16} for similar ideas in the context of Korteweg-de Vries and nonlinear Schr\"odinger equations, respectively). The proposed novel class of integrators is in particular asymptotic consistent and converges in the high-plasma frequency limit, i.e., for~$c \to \infty$, towards solutions of the corresponding Zakharov limit system \eqref{zak}.

Note that in  contrast to previous works (e.g., \cite{BaoZ,BFS17,ChC,HerrS17}) the construction of the numerical schemes and their corresponding error analysis is here  much more challenging due to the the resonant coupling of the highly oscillatory parts triggered by the critical high-plasma frequency~$c$ and the loss of derivative stemming from the nonlinear wave coupling.   \\

For practical implementation issues, we will consider periodic boundary conditions, i.e., consider $z$ and $n$ as functions defined on $(t,x)\in \mathbb{R}\times \mathbb{T}^d$ with values in $\mathbb{R}$, a finite time interval~$(0,T)$ and smooth initial values. Note that our result can be easily extended to complex valued solutions $z(t,x) \in \mathbb{C}$, however, for clarity of presentation we restrict our attention to the real setting. Furthermore, in the following we assume that $r>d/2$ and, to simplify notation, we restrict our attention to  dimensions $d\leq 3$. We denote by $\Vert \cdot \Vert_r$ the standard $H^r= H^r(\mathbb{T}^d)$ Sobolev norm and exploit the well-known bilinear estimate
\begin{equation}\label{bil}
\Vert f g \Vert_r \leq K_{r,d} \Vert f \Vert_r \Vert g \Vert_r
\end{equation}
which holds for some constant $K_{r,d}>0$.

\section{Reformulation of the system}
For a given $c>0$ we define the following operator \emph{(Japanese bracket)}
\begin{equation*}
\nab = \sqrt{-\Delta+c^2}.
\end{equation*}
Note that $\nab$ is well defined since $-\Delta$ is a positive operator. Furthermore, it is a local operator considered on $\mathbb{T}^d$ with its $k-$th Fourier coefficient given by
\begin{equation}\label{def:cnabF}
\left(\nab\right)_k = \sqrt{|k|^2+c^2}.
\end{equation}
Next  we rewrite the Klein-Gordon part $z$ in \eqref{eq:kgzOr} as a first-order system in time via the transformation (see, e.g., \cite{LWP1a,LWP1b})
\begin{equation}
\begin{aligned} \label{eq:ansatz}
u &= z - ic^{-1} \nab^{-1} \partial_t z
\end{aligned}
\end{equation}
such that as $z(t,x) \in\mathbb{R}$ we have
\begin{equation} 
\begin{aligned}\label{z}
z=\frac{1}{2}(u+\overline{u}) .
\end{aligned}
\end{equation}
With the transformation \eqref{eq:ansatz} at hand, the original problem \eqref{eq:kgzOr} can be rewritten as follows
\begin{equation}
\begin{aligned}\label{kgz}
i \partial_t u & = - c \nab u  -  \frac12 c\nab^{-1} n (u+\overline u),\\
 \partial_{tt}{n} & = \Delta n + \frac14 \Delta \vert u+ \overline u\vert^2.
\end{aligned}
\end{equation}

\begin{rem}[Nonlinear coupling]  Note that the coupling in the Klein-Gordon and wave part is driven by the operator $c \nab^{-1}$ and $\nabo$, respectively. With the aid of the Fourier expansion we easily see that the   coupling operator  $c \nab^{-1} \times \nabo$ satisfies
\begin{align*}
\left \Vert c \nab^{-1} \nabo f \right \Vert_r^2 = \sum_{k \in \mathbb{Z}} 
\left \vert \frac{c k }{\sqrt{c^2+k^2}} \right \vert^2\vert \hat{f}_k\vert^2
\end{align*}
which implies that
\[
\left \Vert c \nab^{-1} \nabo f \right \Vert_r^2 \leq c \Vert f \Vert_{r} \qquad \text{as well as} \qquad
\left \Vert c \nab^{-1} \nabo f \right \Vert_r^2 \leq  \Vert f \Vert_{r+1}.
\]
From the first bound we can easily deduce that no loss of derivative occurs when $c=\mathcal{O}(1)$, and hence the KGZ system \eqref{eq:kgzOr} can be solved  much more easily in the low-plasma frequency regime $c =1$. However, standard techniques fail in the high-plasma frequency regime $c \gg 1$ due to the loss of derivative highlighted through the second bound.
\end{rem}

To overcome this \emph{loss of derivative } in the high-plasma frequency regime we pursue the following strategy: Inspired by the numerical analysis of the Zakharov system given in \cite{HerrS17}, see also~\cite{OzaT92} for the original idea in context of the local wellposedness analysis of the Zakharov system, we introduce the new variable
\[
F = \partial_t u
\]
and will further look at  \eqref{kgz} as a system in $(u,\partial_t u ,n,\partial_t n) = (u,F,n,\dot{n})$. Note that with this notation at hand the equation in $u$ given in \eqref{kgz} can be expressed as follows
\begin{equation}\label{0u}
c \nab u  = - i F  -  \frac12 c \nab^{-1} n (u+\overline u).
\end{equation}
Furthermore, taking the time derivative in the first line of \eqref{kgz} yields by the product formula that
\begin{equation}
\begin{aligned}\label{F0}
& i \partial_t F = - c \nab F    - \frac12 c \nab^{-1} \Big\{ \partial_t n (u+\overline u) + n \partial_t (u+\overline u)\Big\}.
\end{aligned}
\end{equation}
As $n$ is real valued we have that
\begin{align*}
\partial_t u & = + i c\nab u   + i \frac12 c\nab^{-1} n (u+\overline u),\\
\partial_t \overline u & = - i c \nab \overline u   - i \frac12 c\nab^{-1} n (u+\overline u)
\end{align*}
which implies that
\begin{equation}\label{uu}
\partial_t (u+\overline u) = i c \nab (u-\overline u).
\end{equation}
Plugging  \eqref{uu} into \eqref{F0} yields with the notation $\partial_t n = \dot n$ that
\begin{equation}
\begin{aligned}\label{F1}
& i \partial_t F = - c \nab F   -\frac12 c \nab^{-1} \Big\{ \dot{n} (u+\overline u) +   i n c \nab (u-\overline u) \Big\}.
\end{aligned}
\end{equation}
System \eqref{kgz} together with equation \eqref{0u} and  \eqref{F1}  thus takes the form
\begin{equation}\label{kgzS}
\begin{aligned}
  i \partial_t F  &= - c \nab F   - \frac12 c \nab^{-1} \Big\{ \dot{n} (u+\overline u) +   i  n c \nab (u-\overline u) \Big\},\\
   \partial_{tt}{n}  &= \Delta n + \frac14 \Delta \vert u+ \overline u\vert^2,\\
    u   &= (c \nab)^{-1}\left\{ - i F   - \frac12 c\nab^{-1} n \left( u(0) + \int_0^t F(\xi) \mathrm{d}\xi + \overline{u(0) + \int_0^t  F(\xi) \mathrm{d}\xi}\right)  \right\}.
\end{aligned}
\end{equation}
Thereby, we use that $c \nab$ is invertible for all $c \neq 0$ as well as the representation
\begin{equation}
u(t) = u(0) + \int_0^t F(\xi) \mathrm{d}\xi.
\end{equation}

\section{Construction of the first order scheme}
In this section we develop a first order \emph{uniformly accurate numerical scheme} which allows us to approximate solutions of the KGZ system \eqref{kgz} with order $\mathcal{O}(\tau)$ uniformly in the parameter $c$. Our  approach is thereby based on looking at the reformulated system \eqref{kgzS} and approximating  the corresponding Duhamel's formula in $(F,n,\partial_t n = \dot{n})$. However, and in great difference to classical exponential and trigonometric integration techniques (cf., e.g., \cite{Gau15,HerrS17,HLW,HochOst10}), we will carefully treat the highly oscillatory phases triggered by the plasma frequency $c$ in an exact way.

Duhamel's formula in $(F,n,\dot{n})$ reads (see \eqref{kgzS})
\begin{equation}
\begin{aligned}\label{Duh}
F(t_\ell+\tau) & = \e^{i \tau c \nab}F(t_\ell)  \red+ \frac{i}{2}\black c \nab^{-1} \int_0^\tau  \e^{i (\tau-\xi) c \nab}
\Big\{ \dot{n}(t_\ell+\xi) \big(u(t_\ell+\xi)+\overline u(t_\ell+\xi)\big) \\
 & \qquad \qquad\qquad \qquad \qquad \qquad\qquad +i    n(t_\ell+\xi)  c \nab \big(u(t_\ell+\xi)-\overline u(t_\ell+\xi)\big) \Big\}\dd \xi,\\
n(t_\ell+\tau)  &= \coso n(t_\ell) + \nabo^{-1} \sino  \dot{n}(t_\ell) \\
& + \frac14 \nabo^{-1}\int_0^\tau \sin((\tau-\xi)\nabo) \Delta 
\big\vert u(t_\ell+\xi) + \overline u(t_\ell+\xi) \big\vert^2 \dd \xi,\\
 \dot{n}(t_\ell+\tau)  &= - \nabo \sino n(t_\ell) + \coso  \dot{n}(t_\ell) \\
& + \frac14\int_0^\tau \cos((\tau-\xi)\nabo) \Delta 
\big\vert u(t_\ell+\xi) + \overline u(t_\ell+\xi) \big\vert^2 \dd \xi.
\end{aligned}
\end{equation}
Furthermore, observe that for $u$ we have (see \eqref{kgz})
\begin{equation}\label{DuhU}
u(t_\ell+\tau)  = \e^{i \tau c \nab}u(t_\ell)    \red+ \frac{i}{2}\black  c \nab^{-1} \int_0^\tau  \e^{i (\tau-\xi) c \nab}
n(t_\ell+\xi) \big(u(t_\ell+\xi)+\overline u(t_\ell+\xi)\big) \dd \xi.
\end{equation}
\begin{rem}\label{rem:uT}
Note that $
\partial_t u = F = \mathcal{O}(c\nab).$ 
Thus, if we would approximate the integrals in~\eqref{Duh}  by employing the classical Taylor series expansion
\[
u(t_\ell+\xi) = u(t_\ell) + \mathcal{O}(\xi \partial_t u) = u(t_\ell) + \mathcal{O}(\xi c\nab )
\]
(or a classical quadrature formula) this would yield a local error of order $\mathcal{O}(\tau^2 c^2)$ and, in particular, not the aimed uniform approximation property. Henceforth, standard exponential integrator techniques (cf. \cite{HochOst10}) fail, and a more careful approximation technique has to be applied.
\end{rem}

\subsection{Collection of essential lemma and notation}
We start off by collecting some useful lemma which will be essential  in the derivation of uniform approximations with respect to $c$. Thereby we will in particular exploit the following refined bilinear estimates: For $\sigma_1+\sigma_2\geq 0$  (and as we assume that $1\leq d \leq 3$)  it holds that
\begin{equation*}\label{bil2}
\begin{aligned}
& \Vert f g \Vert_{\sigma} \leq K_{r,d} \Vert f\Vert_{\sigma_1} \Vert g \Vert_{\sigma_2} \quad \text{ for  all } \sigma \leq \sigma_1+\sigma_2-\textstyle\frac{d}{2}\quad \text{ with } \sigma_1,\sigma_2 \text{ and } -\sigma \neq\textstyle \frac{d}{2},\\
&\Vert f g \Vert_{\sigma} \leq K_{r,d} \Vert f\Vert_{\sigma_1} \Vert g \Vert_{\sigma_2} \quad \text{  for  all } \sigma < \sigma_1+\sigma_2-\textstyle\frac{d}{2}\quad \text{ with } \sigma_1,\sigma_2 \text{ or } -\sigma =\textstyle \frac{d}{2}.
\end{aligned}
\end{equation*}
In particular, by setting $\sigma = \sigma_1 = r-1$  and $\sigma_2 = r$ we can thus conclude  that
\begin{equation}\label{bil2}
\Vert f g \Vert_{r-1} \leq K_{r,d} \Vert f \Vert_{r-1} \Vert g \Vert_{r},
\end{equation}
where we use that $\sigma_2 = r >d/2$ as well as $\sigma_1+\sigma_2 = 2r -1 > 0$.

\begin{lem}\label{cnab}For all $t \in \mathbb{R}$ and $c\neq0$ we have that 
\begin{equation}\label{lcnab}
\begin{aligned}
& \left\lVert c\nab^{-1} f\right\rVert_r \leq \left\lVert f \right\rVert_{r},\quad \Vert \mathrm{e}^{i t c\nab} f \Vert_r  = \Vert f \Vert_r, \quad
 \left \Vert (c \nab - c^2 ) f \right \Vert_r \leq \frac12 \Vert f \Vert_{r+2},\\
& \Vert ( \e^{-i \xi c \nab} - \e^{-i \xi c^2})f \Vert_r  \leq \frac12 \xi \Vert f \Vert_{r+2}, \\
&\Vert \nab^{-2} f \Vert_{r} \leq \mathrm{min}\left(\frac{1}{c}\Vert f \Vert_{r-1}, 
\Vert f \Vert_{r-2}, \frac{1}{c^2} \Vert f \Vert_r\right), \quad\Vert \nab^{-2} \left( f c \nab g \right)\Vert_{r-1}\leq K \Vert f \Vert_{r-1} \Vert g \Vert_{r+1}.
\end{aligned}
\end{equation} 
\end{lem}
\begin{proof} The estimates in the first and second row follow from \cite{BFS17}. Furthermore, observe that
\[
 \frac{1}{c^2+k^2} \leq \mathrm{min}\left(\frac{1}{c \vert k \vert},\frac{1}{k^2},\frac{1}{c^2}\right)
\qquad \text{and} \qquad 
\frac{  \sqrt{c^2+k^2}}{c} \leq \frac{c +  \vert k \vert}{c} \leq  1+ c^{-1}\vert k \vert .
\]
The second inequality together with the bilinear estimate \eqref{bil2} and the definition of $\nab$ in Fourier space (see \eqref{def:cnabF}) in particular implies that
\begin{align*}
\left\Vert \nab^{-2} \left( f c \nab g \right)\right\Vert_{r-1}
&\leq \left \Vert f  \frac{c\nab}{c^2} g \right \Vert_{r-1} \leq K \Vert f  \Vert_{r-1}\left \Vert   \frac{\nab}{c} g\right\Vert_{r}\leq K \Vert f \Vert_{r-1} \Vert g \Vert_{r+1}
\end{align*}
for some constant $K>0$ independent of $c$. This concludes the estimates in the last row of~\eqref{lcnab}.
\end{proof}

In the following we set
\begin{equation}\label{Mr}
M_{T,r} = \mathrm{max}\left(\sup_{0 \leq t \leq T}\left\{\Vert u(t)\Vert_{r+1}+  \left \Vert (c\nab^{-1}) F(t) \right \Vert_{r+1} +  \Vert n(t)\Vert_{r}+  \Vert \dot n(t)\Vert_{r-1}\right\},1\right) 
\end{equation}
and introduce a suitable definition for the occuring remainders.
\begin{defn}[Remainder]\label{def:rem} We will denote all constants which can be chosen independently of $c$ by $K$. Furthermore, we write
\begin{equation}\label{rem}
f = g + \mathcal{R}_{r+s}\qquad \text{if}\qquad
\Vert f- g \Vert_r \leq K M_{T,r+s}^p
\end{equation}
for some $p\in \mathbb{N}$ and $K>0$ independent of $c$.
\end{defn}
We will also make use of the so-called $\varphi-$functions.
\begin{defn}[The $\varphi-$ functions, see \cite{HochOst10}]\label{Def:phifunc}
For $\zeta \in \mathbb{C}$ we set
\begin{equation}\label{def:phifunc}
\begin{aligned}
& \tau \varphi_1( \zeta \tau) := \int_0^\tau \e^{ \zeta \xi}\dd \xi =  \tau\frac{\e^{ \zeta \tau} - 1}{ \zeta \tau }\quad \text{and}\quad  \tau^2 \varphi_2( \zeta \tau) := \int_0^\tau \xi  \varphi_1(\zeta \xi)\dd \xi = \tau^2 \frac{\varphi_1( \zeta \tau)-1}{\zeta\tau}.
\end{aligned}
\end{equation}
Furthermore, we set
\[
\tau^2 \Psi_2(\tau\zeta) := \int_0^\tau \xi \mathrm{e}^{ \xi \zeta}\dd \xi = \tau^2 \frac{\varphi_0(\tau\zeta) - \varphi_1(\tau\zeta)}{\tau\zeta}.
\]
\end{defn}
The following lemma will allow us to carry out a classical Taylor series expansions in $n(t_\ell+\xi)$ and $\dot n(t_\ell+\xi)$ in the construction of our numerical scheme without producing remainders which depend on $c$.
\begin{lem}\label{lem:appn} For all $\xi \in \mathbb{R}$ it holds that
\begin{align}\label{eq:appn}
 & \Vert n(t_\ell+\xi) - n(t_\ell) \Vert_r  + \Vert \dot{n}(t_\ell+\xi) - \dot{n}(t_\ell)\Vert_{r-1}\leq \vert \xi \vert K M_{r+1}^2
\end{align}
for some constant $K>0$ which can be chosen independently of $c$ such that in particular
\begin{equation}\label{Tay:nnp}
n(t_\ell+\xi) = n(t_\ell) +  \xi \mathcal{R}_{r+1},\qquad \dot n(t_\ell+\xi) = \dot n(t_\ell) +  \xi \mathcal{R}_{r+2}.
\end{equation}
\end{lem}
\begin{proof}
Duhamel's formula \eqref{Duh} in $(n,\dot{n})$ yields
\begin{align*}
& n(t_\ell+\xi) - n(t_\ell) = \big(\mathrm{cos}(\xi \nabo)-1\big) n(t_\ell) + \xi \mathrm{sinc}(\xi \nabo) \dot{n}(t_\ell) + \xi \mathcal{R}_{r+1},\\
& \dot{n}(t_\ell+\xi) - \dot{n}(t_\ell) = \big(\mathrm{cos}(\xi \nabo)-1\big) \dot{n}(t_\ell)
- \xi \mathrm{sinc}(\xi\nabo)\nabo^2 n(t_\ell) + \xi \mathcal{R}_{r+2}.
\end{align*}
The assertion thus follows by the estimates
\begin{equation}\label{appsin1}
\begin{aligned}
& \Vert (\mathrm{cos}(\xi\nabo) - 1)f\Vert_r+\Vert (\mathrm{sinc}(\xi \nabo) -1 )f\Vert_r\leq 3\xi^2 \Vert f \Vert_{r+2}
\end{aligned}
\end{equation}
together with the bilinear estimate  \eqref{bil}.
\end{proof}
 In the approximation of $u$, however, we need to be much more careful as a classical Taylor series expansion would lead to
\[
u(t_\ell+\xi) = u(t_\ell) + \xi c \nab \mathcal{R}_{r},
\]
and  trigger an error at order $\mathcal{O}(\xi c^2)$ (see also Remark \ref{rem:uT}).
\begin{lem}\label{lem:appu}
For all $\xi \in \mathbb{R}$ it holds that
\begin{equation}\label{eq:appu}
\Vert u(t_\ell+\xi) - \e^{i c^2 \xi} u(t_\ell)\Vert_r + \Vert u(t_\ell+\xi) - \e^{i c\nab \xi} u(t_\ell)\Vert_r \leq \vert \xi \vert K \left(M_{r+2}+ M_r^2\right)
\end{equation}
for some constant $K>0$ which can be chosen independently of $c$ such that in particular
\begin{equation}\label{Tay:u}
\begin{aligned}
&u(t_\ell+\xi) = \e^{ic\nab \xi} u(t_\ell) +  \xi \mathcal{R}_{r}\quad \text{and}\quad u(t_\ell+\xi) = \e^{ic^2\xi} u(t_\ell) +  \xi \mathcal{R}_{r+2}.
\end{aligned}
\end{equation}
\end{lem}
\begin{proof}
Duhamel's formula in $u$ (see \eqref{DuhU})  implies that

\[
u(t_\ell+\xi)  =\e^{i \xi c \nab} u(t_\ell) + \xi \mathcal{R}_{r}
\]
which yields the first assertion. Furthermore, we can write
\[
u(t_\ell+\xi) - \e^{ic^2 \xi} u(t_\ell) = \left(\e^{i \xi c \nab}-\e^{i c^2\xi}\right) u(t_\ell) + \xi \mathcal{R}_{r}.
\]
Together with Lemma \ref{cnab}  this concludes the second assertion.
\end{proof}

\subsection{A first order uniformly accurate oscillatory integrator for the KGZ system}

\subsubsection{Approximation in $F$}
Recall Duhamel's formula in $F$  (see \eqref{Duh})
\begin{equation}
\begin{aligned}
F(t_\ell+\tau) & = \e^{i \tau c \nab}F(t_\ell) \\&
  \red+ \frac{i}{2}\black  c \nab^{-1} \int_0^\tau  \e^{i (\tau-\xi) c \nab}
\Big\{ \dot{n}(t_\ell+\xi) \big(u(t_\ell+\xi)+\overline u(t_\ell+\xi)\big) \\
 & \qquad \qquad\qquad \qquad+i    n(t_\ell+\xi)  c \nab \big(u(t_\ell+\xi)-\overline u(t_\ell+\xi)\big) \Big\}\dd \xi.
\end{aligned}
\end{equation}
Multiplying the above formula with the operator $(c\nab)^{-1}$ and employing the expansions for~$(n,\dot{n})(t_\ell+\xi)$ given \eqref{Tay:nnp} and for $u(t_\ell+\xi)$ given in \eqref{Tay:u} we obtain by Lemma \ref{cnab}
that
\begin{equation*}
\begin{aligned}
(c\nab)^{-1} F(t_\ell+\tau) & = \e^{i \tau c \nab}(c\nab)^{-1}  F(t_\ell)
 \\&     \red+ \frac{i}{2}\black  \nab^{-2} \int_0^\tau  \e^{i (\tau-\xi) c \nab}
\Big\{ \dot{n}(t_\ell) \big(\e^{i c^2\xi}u(t_\ell)+\e^{-i c^2\xi}\overline u(t_\ell)\big) \\
 & +i    n(t_\ell)  c \nab \big(\e^{i c^2\xi} u(t_\ell)-\e^{-i c^2\xi}\overline u(t_\ell)\big) \Big\}\dd \xi \\&+ \tau^2 \mathcal{R}_{r+2}\\
& = \e^{i \tau c \nab}(c\nab)^{-1}  F(t_\ell) \\&     \red+ \frac{i}{2}\black  \nab^{-2} \e^{i\tau  c \nab}
 \int_0^\tau  
\Big\{\e^{i \xi \frac12 \Delta  } \big(\dot{n}(t_\ell)  u(t_\ell)\big)+\e^{-i \xi (c\nab+c^2)} \big(\dot{n}(t_\ell)  \overline u(t_\ell)\big)\\
 & +i   \e^{i \xi \frac12 \Delta  } \big(n(t_\ell)  c \nab u(t_\ell)\big)-i\e^{-i \xi (c\nab+c^2)}  \big( n(t_\ell) c \nab  \overline u(t_\ell)\big)\Big\}\dd \xi\\&+ \tau^2 \mathcal{R}_{r+2}.
\end{aligned}
\end{equation*}
With the definition of the $\varphi_1$ function in \eqref{def:phifunc} we furthermore obtain that
\begin{equation}\label{AppF}
\begin{aligned}
(c\nab)^{-1} F(t_\ell+\tau) 
  & = \e^{i \tau c \nab}(c\nab)^{-1}  F(t_\ell)\\&
      \red+ i \black \frac{\tau}{2}  \nab^{-2} \e^{i\tau  c \nab} \varphi_1\left(\textstyle \frac{i}{2} \tau \Delta \right) \Big(\dot{n}(t_\ell)  u(t_\ell) + i n(t_\ell)  c \nab u(t_\ell)\Big)\\
  &     \red+ i \black\frac{\tau}{2}  \nab^{-2} \e^{i\tau  c \nab} \varphi_1\left(-i \tau (c\nab+c^2)\right)\Big(\dot{n}(t_\ell) \overline u(t_\ell)- i n(t_\ell)  c \nab \overline u(t_\ell)\Big)\\& + \tau^2 \mathcal{R}_{r+2}.
\end{aligned}
\end{equation}

\subsubsection{Approximation in $u$}
Recall that at time $t =t_\ell$ we have that (see \eqref{kgzS})
\begin{equation}\label{up}
u(t_\ell) = (c \nab)^{-1}\left\{ - i F(t_\ell)   - \frac12 c\nab^{-1} n(t_\ell) \left( I^F(t_\ell) + \overline{I^F(t_\ell)}\right) \right\},
\end{equation}
where we have set
\[
I^F(t_\ell) := u(0) + \int_0^{t_\ell} F(\xi)\dd\xi =  u(0) + \sum_{k=0}^{\ell-1} \int_0^\tau F(t_k+\xi) \dd \xi .
\]
To obtain an approximation to $I^F(t_\ell)$ we will use the approximation \eqref{AppF} which yields that
\begin{equation}
\begin{aligned}\label{appIF}
I^F(t_\ell) & = u(0) + \sum_{k=0}^{\ell -1} \int_0^\tau \left( \e^{i \xi c \nab} F(t_k) + c^2 \xi\left(\mathcal{R}_{r+2}\right)_k\right) \dd \xi\\
& = u(0) + \left(\tau \sum_{k=0}^{\ell-1} \varphi_1(i \tau c \nab)F(t_k) \right)+ \tau c^2 t_\ell \mathcal{R}_{r+2}.
\end{aligned}
\end{equation}
In the following we set
\begin{equation}\label{defSF}
S^F(t_\ell) = u(0) + \tau \sum_{k=0}^{\ell-1} \varphi_1(i \tau c \nab)F(t_k).
\end{equation}
Then, plugging the approximation \eqref{appIF} into \eqref{up} yields thanks to the estimate $\Vert \nab^{-2} c^2 \Vert_r \leq 1$ (see Lemma \ref{cnab}) that
\begin{equation}\label{Appu}
\begin{aligned}
u(t_\ell) & = (c \nab)^{-1}\left\{ - i F(t_\ell)   - \frac12 c\nab^{-1} n(t_\ell) \left( S^F(t_\ell) + \overline{S^F(t_\ell)}\right) \right\} \\&+ \frac12  \nab^{-2}\left\{ n(t_\ell) \tau c^2   \mathcal{R}_{r+2}\right\}\\
& =  (c \nab)^{-1} \left\{ - i F(t_\ell)   - \frac12 c\nab^{-1} n(t_\ell) \left( S^F(t_\ell) + \overline{S^F(t_\ell)}\right) \right\} \\
& + \tau \mathcal{R}_{r+2}.
\end{aligned}
\end{equation}

\subsubsection{Approximation in $(n,\dot{n})$}
Recall Duhamel's formula in $(n,\dot{n})$ (cf. \eqref{Duh})
\begin{equation*}
\begin{aligned}
n(t_\ell+\tau)  &= \coso n(t_\ell) + \nabo^{-1} \sino \dot{n}(t_\ell) \\
& + \frac14 \nabo^{-1}\int_0^\tau \sin((\tau-\xi)\nabo) \Delta 
\big\vert u(t_\ell+\xi) + \overline u(t_\ell+\xi) \big\vert^2 \dd \xi,\\
\dot{n}(t_\ell+\tau)  &= - \nabo \sino n(t_\ell) + \coso \dot{n}(t_\ell) \\
& + \frac14\int_0^\tau \cos((\tau-\xi)\nabo) \Delta 
\big\vert u(t_\ell+\xi) + \overline u(t_\ell+\xi) \big\vert^2 \dd \xi.
\end{aligned}
\end{equation*}
Employing the approximation of $u(t_\ell+\xi)$ given in \eqref{Tay:u}  together with the trignomeric  approximations
\begin{equation}\label{appsin}
\begin{aligned}
&\Vert (\mathrm{sin}(\xi \nabo) - \xi \nabo )f\Vert_r + \Vert (\mathrm{cos}(\xi\nabo) - 1)f\Vert_r+\Vert (\mathrm{sinc}(\xi \nabo) -1 )f\Vert_r\leq 3\xi^2 \Vert f \Vert_{r+2}
\end{aligned}
\end{equation}
we obtain that
\begin{equation*}
\begin{aligned}
n(t_\ell+\tau)  &= \coso n(t_\ell) + \nabo^{-1} \sino \dot{n}(t_\ell) \\
& + \frac14 \nabo^{-1}\sinco \int_0^\tau \big((\tau-\xi) \nabo \big)\Delta 
\left\{2 \vert u(t_\ell)\vert^2 + \e^{2i c^2\xi} u(t_\ell)^2 + \e^{-2ic^2\xi} \overline{u(t_\ell)}^2
\right\} \dd \xi,\\
& + \tau^3 \mathcal{R}_{r+4},\\
\dot{n}(t_\ell+\tau)  &= - \nabo \sino n(t_\ell) + \coso \dot{n}(t_\ell) \\
& + \frac14 \coso \int_0^\tau \Delta 
\left\{2 \vert u(t_\ell)\vert^2 + \e^{2i c^2\xi} u(t_\ell)^2 + \e^{-2ic^2\xi} \overline{u(t_\ell)}^2
\right\} \dd \xi\\
& + \tau^2 \mathcal{R}_{r+4}.
\end{aligned}
\end{equation*}
Together with the definition of the $\varphi_1$ and $\varphi_2$ function (see \eqref{def:phifunc}) we thus derive that
\begin{equation}\label{Appnnp}
\begin{aligned}
n(t_\ell+\tau)  &= \coso n(t_\ell) + \nabo^{-1} \sino \dot{n}(t_\ell) \\
& + \frac{\tau^2}{4} \sinco \Delta 
\left\{  \vert u(t_\ell)\vert^2 + \varphi_2(2i c^2\tau) u(t_\ell)^2 + \varphi_2(-2ic^2\tau) \overline{u(t_\ell)}^2
\right\}\\
& + \tau^3 \mathcal{R}_{r+4},\\
\dot{n}(t_\ell+\tau)  &= - \nabo \sino n(t_\ell) + \coso \dot{n}(t_\ell) \\
& + \frac{\tau}{4} \coso \Delta 
\left\{2 \vert u(t_\ell)\vert^2 + \varphi_1(2i c^2\tau) u(t_\ell)^2 +  \varphi_1(-2i c^2\tau) \overline{u(t_\ell)}^2
\right\}\\
& + \tau^2 \mathcal{R}_{r+4}.
\end{aligned}
\end{equation}

\subsubsection{A uniformly accurate oscillatory  integrator of first order}
Collecting the approximations in~\eqref{AppF},~\eqref{Appu} (together with \eqref{defSF}) and \eqref{Appnnp} motivate us to define our numerical scheme as follows:

For $0 \leq \ell \leq n-1$ we set
\begin{equation}\label{scheme}
\begin{aligned}
(c\nab)^{-1} F_{\ell+1} & = \e^{i \tau c \nab}(c\nab)^{-1}  F_\ell\\&
       \red+ i \black \frac{\tau}{2}  \nab^{-2} \e^{i\tau  c \nab}  \varphi_1\left(\textstyle \frac{i}{2} \tau \Delta \right) \Big(\dot{n}_\ell  u_\ell + i n_\ell  c \nab u_\ell\Big)\\
  &        \red+ i \black \frac{\tau}{2}  \nab^{-2} \e^{i\tau  c \nab} \varphi_1\left(-i \tau (c\nab+c^2)\right)\Big(\dot{n}_\ell \overline u_\ell -  i n_\ell c \nab \overline u_\ell\Big)\\
  n_{\ell+1}  &= \coso n_\ell+ \nabo^{-1} \sino \dot{n}_\ell \\
& + \frac{\tau^2}{4} \sinco \Delta 
\left\{  \vert u_\ell \vert^2 + \varphi_2(2i c^2\tau) u_\ell ^2 + \varphi_2(-2ic^2\tau) \overline{u_\ell }^2
\right\}\\
\dot{n}_{l+1} &= - \nabo \sino n_\ell + \coso \dot{n}_\ell \\
& + \frac{\tau}{4} \coso \Delta 
\left\{2 \vert u_\ell \vert^2 + \varphi_1(2i c^2\tau) u_\ell ^2 +  \varphi_1(-2i c^2\tau) \overline{u_\ell }^2
\right\}\\
S^F_{\ell+1} & = S^F_\ell + \tau \varphi_1(i \tau c \nab)F_{\ell+1}\\
u_{\ell+1} & = c^{-1} \nab^{-1}\left\{ - i F_{\ell+1} -  \frac12 c\nab^{-1} n_{\ell+1} \left( S^F_{\ell+1} +\overline{S^F_{\ell+1}}\right) \right\}
\end{aligned}
\end{equation}
and choose the initial values (cf. \eqref{0u})
\begin{equation}\label{initial}
\begin{aligned}
& u_0 := u(0), \qquad n_0:=n(0),\qquad \dot{n}_0 := \partial_t n(0),\\
& F_0 := i c \nab u_0      \red+ \frac{i}{2}  \black c\nab^{-1} n_0 (u_0+\overline{u_0}), \\& S_0^F := u_0 +  \tau  \varphi_1(i \tau c \nab)F_{0}.
\end{aligned}
\end{equation}
\begin{rem}
Note that for practical implementation issues we may write
\begin{align*}
S^F_{\ell+1} & = S^F_\ell + \tau \varphi_1(i \tau c \nab)F_{\ell+1}\\
& =  S^F_\ell  - i \left( \e^{i \tau c \nab}-1\right) (c\nab)^{-1} F_{\ell+1}
\end{align*}
thanks to the definition of the $\varphi_1-$function given in  \eqref{def:phifunc}. \Red Note that the calculation of the inverse $(c\nab)^{-1}$ is thereby not necessary as $(c\nab)^{-1}F_{\ell+1}$ itself is computed in the iterative scheme \eqref{scheme}.\black
\end{rem}
In the next section we carry out the convergence analysis of the scheme \eqref{scheme}.
\section{Convergence analysis of the first order scheme}\label{sec:con1}

In the following we set (cf. \eqref{Mr})
\begin{equation}\label{Mn}
B_{t_\ell,r} = \mathrm{max}\left(\sup_{0 \leq k \leq \ell } \left\{\Vert u_k \Vert_{r+1} + \Vert (c\nab)^{-1} F_k\Vert_{r+1}+ \Vert n_k \Vert_{r}+  \Vert \dot{n}_k \Vert_{r-1}\right\},1\right) .
\end{equation}

\subsection{Error in $F$}
Taking the difference of  $F(t_\ell+\tau)$ given in \eqref{AppF}  and $F_{\ell+1}$ defined in~\eqref{scheme} we readily obtain thanks to the error bounds on the remainders \eqref{rem} (see Definition \ref{def:rem}), the bilinear estimate \eqref{bil} and the isometric property $\Vert \e^{i \tau c \nab} f\Vert_r = \Vert f \Vert_r$ (see Lemma \ref{cnab}) that
\begin{equation}\label{errF1}
\begin{aligned}
\big\Vert (c\nab)^{-1}& \left(F(t_\ell+\tau)-F_{\ell+1}\right)\big\Vert_{r+1}\\& \leq  
 \left \Vert (c\nab)^{-1} \left(F\tl-\F\right)\right\Vert_{r+1}\\
&+  \tau \left\Vert  \nab^{-2} (\dot{n}\tl u\tl - \np \u) \right\Vert_{r+1}\\
&+ \tau \left \Vert \nab^{-2}   \varphi_1\left(\textstyle \frac{i}{2} \tau \Delta \right)  \Big (n\tl c\nab u\tl -  n_\ell  c \nab u_\ell \Big) \right\Vert_{r+1}\\
& + \tau \left \Vert \nab^{-2}  \varphi_1\left(-i \tau (c\nab+c^2)\right)  \Big ( n\tl c\nab \overline u\tl - n_\ell  c \nab \overline u_\ell\Big) \right\Vert_{r+1}\\
& + \tau^2 K M_{t_{\ell+1},r+3}^p\\
& =:  \left \Vert (c\nab)^{-1} \left(F\tl-\F\right)\right\Vert_{r+1} + T^F_1 + T^F_2 + T^F_3 + \tau^2 K M_{t_{\ell+1},r+3}^p,
\end{aligned}
\end{equation}
where we have used that the definition of the $\varphi_1$ function (see  \eqref{def:phifunc}) implies that
 \[
\left\Vert   \varphi_1\left(\textstyle \frac{i}{2} \tau \Delta \right) \right\Vert_{r} \leq 1\quad \text{and} \quad \Vert \varphi_1\left(-i \tau (c\nab+ c^2)\right) \Vert_r \leq 1.
\]
We will estimate the terms on the right hand side $T_j^F$ separately.
\subsubsection{Bound on the first term $T^F_1$.} Note that
\begin{align*}
 \dot{n}\tl u\tl - \np \u & = \dot{n} \tl u \tl - (\np - \dot{n} \tl + \dot{n}\tl) \u\\
& = (\dot{n}\tl - \np) \u + \dot{n}\tl (u\tl - \u).
\end{align*}
Thanks to Lemma \ref{cnab} and the bilinear estimate \eqref{bil2} we have that
\begin{equation*}
\Vert \nab^{-2} (f g) \Vert_{r+1} \leq \Vert f g \Vert_{r-1} \leq  K \Vert f \Vert_{r-1} \Vert g \Vert_{r+1}
\end{equation*}
which thus implies that
\begin{equation}\label{b1}
\begin{aligned}
T^F_1 &:= \tau \left\Vert  \nab^{-2} (\dot{n}\tl u\tl - \np \u) \right\Vert_{r+1} \\
&\leq \tau K \left(\Vert \dot{n}\tl - \np\Vert_{r-1} \Vert u_\ell\Vert_{r+1} +\Vert \dot{n}\tl\Vert_{r-1} \Vert u\tl - \u\Vert_{r+1}   \right) 
 \\
& \leq 
\tau K  M_{t_\ell,r} \B \left(\Vert \dot{n}\tl - \np\Vert_{r-1} + \Vert u\tl - \u\Vert_{r+1}   \right) 
\end{aligned}
\end{equation}
where $M_{t_\ell,r}$ and $B_{t_\ell,r}$ are defined in \eqref{Mr} and \eqref{Mn}, respectively.

The second and third term have to be bounded more carefully.
\subsubsection{Bound on the second term $T_2^F$.}  Note that for $\zeta \in \mathbb{R}$ with $\zeta \neq 0$ we have that
\begin{equation}\label{bphi}
 \left\Vert  \tau \varphi_1\left( i \tau \zeta \right) f \right \Vert_{r+1} =  \left
 \Vert  \tau \frac{\e^{i \tau \zeta} - 1}{i \tau \zeta }  f \right\Vert_{r+1} \leq  \Vert (\zeta)^{-1} f \Vert_{r+1}.
\end{equation}
Thanks to the relation
\begin{equation}\label{rnu}
\begin{aligned}
n\tl c \nab u\tl - \n c\nab \u  & = n\tl c \nab u\tl - (\n-n\tl + n \tl) c\nab \u  \\
& = (n\tl - \n) c \nab \u + n\tl c \nab(u\tl - \u)
\end{aligned}
\end{equation}
 we thus obtain that
\begin{equation}\label{b002}
\begin{aligned}
T^F_2& := \tau  \left \Vert  \right. \nab^{-2}   \varphi_1\left(\textstyle \frac{i}{2} \tau \Delta \right) \Big (n\tl c\nab u\tl -  n_\ell  c \nab u_\ell \Big) \left. \right\Vert_{r+1}\\
 & \leq 
 \left \Vert \nab^{-2} \left( \tau \varphi_1\left(\textstyle \frac{i}{2} \tau \Delta \right)\right) \Big ((n\tl - \n)c\nab \u \Big) \right\Vert_{r+1}
\\
&+  \tau \left \Vert \nab^{-2}  \varphi_1\left(\textstyle \frac{i}{2} \tau \Delta \right) \Big (n\tl c\nab (u\tl - \u)\Big) \right\Vert_{r+1}\\
 & \leq
  \left \Vert \nab^{-2}  \Big ((n\tl - \n)c\nab \u \Big) \right\Vert_{r-1}
\\
&+   \tau \left \Vert \nab^{-2} \Big (n\tl c\nab (u\tl - \u)\Big) \right\Vert_{r+1}.
\end{aligned}
\end{equation}
Thanks to the last estimate in  Lemma \ref{cnab} we have that
\begin{align}\label{E1}
 \left \Vert \nab^{-2}  \Big ((n\tl - \n)c\nab \u \Big) \right\Vert_{r-1} \leq K \Vert n\tl - \n\Vert_{r-1} \Vert u_\ell \Vert_{r+1}.
\end{align}
Furthermore, the estimate
$
\frac{ c |k |}{c (c + |k|)} \leq 1
$ 
together with the definition of the operator $\nab$ in Fourier space (see \eqref{def:cnabF}) yields that
\[
\left \Vert \frac{c\nab}{c (c+\nabo)} f\right\Vert_{r} \leq \Vert f \Vert_r.
\]
 Lemma \ref{cnab}  together with the above bound implies
 \begin{equation}
\begin{aligned}\label{E2}
& \tau  \left \Vert \nab^{-2} \Big (n\tl c\nab (u\tl - \u)\Big) \right\Vert_{r+1} \\& \leq \tau \left \Vert \nab^{-2} \Big (n\tl \frac{c\nab}{c (c+\nabo)} c (c+\nabo) (u\tl - \u)\Big) \right\Vert_{r+1}\\
& \leq \tau \left \Vert \nab^{-2} \Big (n\tl \frac{c\nab}{c (c+\nabo)} c^2  (u\tl - \u)\Big) \right\Vert_{r+1}
+ \tau \left \Vert \nab^{-2} \Big (n\tl \frac{c\nab}{c (c+\nabo)} c \nabo  (u\tl - \u)\Big) \right\Vert_{r+1}\\
& \leq \tau \frac{1}{c^2}\left \Vert  n\tl   \frac{c\nab}{c (c+\nabo)} c^2  (u\tl - \u) \right\Vert_{r+1}
+ \tau\frac{1}{c} \left \Vert n\tl \frac{c\nab}{c (c+\nabo)} c \nabo  (u\tl - \u)\right\Vert_{r}\\
& \leq \tau K \Vert n\tl  \Vert_{r+1} \Vert u\tl - \u \Vert_{r+1}.
\end{aligned}
\end{equation}
Plugging \eqref{E1} and \eqref{E2} into \eqref{b002} we can thus conclude that
\begin{equation}\label{b2}
\begin{aligned}
T_2^F & \leq K\Big(  \Vert n\tl - \n\Vert_{r-1} \Vert \u  \Vert_{r+1}
+  \tau  \Vert n\tl \Vert_{r+1}  \left \Vert u\tl - \u \right\Vert_{r+1}\Big)\\
& \leq \tau K M_{t_\ell,r+1} B_{t_\ell,r} 
\Big( \frac{1}{\tau} \Vert n\tl - \n\Vert_{r-1} +  \left \Vert u\tl - \u \right\Vert_{r+1}
\Big).
 \end{aligned}
 \end{equation}
\subsubsection{Bound on the third term $T_3^F$.}  Similarly to the bound on $T_2^F$ we obtain by the  relation~\eqref{rnu}  using \eqref{bphi} together with the estimate
\begin{align*}
\left \Vert \frac{1}{c\nab+c^2} f \right \Vert_r \leq \frac{1}{c^2} \Vert f \Vert_r
\end{align*}
 that
\begin{equation*}\label{b02}
\begin{aligned}
T_3^F &:=  \left \Vert \nab^{-2}  \varphi_1\left(-i \tau (c\nab+c^2)\right)  \Big ( n\tl c\nab \overline u\tl - n_\ell  c \nab \overline u_\ell\Big) \right\Vert_{r+1}\\
&  \leq
  \left \Vert \nab^{-2} \frac{1}{ c \nab + c^2} \Big ((n\tl - \n)c\nab \overline \u \Big) \right\Vert_{r+1}
\\&+   \tau \left \Vert \nab^{-2} \Big (n\tl c\nab \overline{(u\tl - \u)}\Big) \right\Vert_{r+1}\\
& \leq   \left \Vert (n\tl - \n)\frac{c\nab}{c^2} \overline \u  \right\Vert_{r-1}
\\&+   \tau \left \Vert \nab^{-2} \Big (n\tl c\nab \overline{(u\tl - \u)}\Big) \right\Vert_{r+1}.
\end{aligned}
\end{equation*}
With similar arguments as above we can thus conclude 
\begin{equation}\label{b3}
\begin{aligned}
T_3^F & \leq  K\Big(  \Vert n\tl - \n\Vert_{r-1} \Vert \u  \Vert_{r+1}
+   \tau \Vert n\tl \Vert_{r+1}  \left \Vert u\tl - \u \right\Vert_{r+1}\Big)\\
& \leq \tau K M_{t_\ell,r+1} B_{t_\ell,r} 
\Big( \frac{1}{\tau} \Vert n\tl - \n\Vert_{r-1} +  \left \Vert u\tl - \u \right\Vert_{r+1}.
\end{aligned}
\end{equation}

\subsubsection{Bound on error in $F$.} Plugging the bounds \eqref{b1}, \eqref{b2} and \eqref{b3}  into \eqref{errF1} yields that
\begin{equation}\label{errF}
\begin{aligned}
\big\Vert (c\nab)^{-1} &\left(F(t_\ell+\tau)-F_{\ell+1}\right)\big\Vert_{r+1} \\ & \leq \left \Vert (c\nab)^{-1} \left(F\tl-\F\right)\right\Vert_{r+1}\\
&  + \tau K  M_{t_l,r+1} \B \Big(\Vert u\tl - \u\Vert_{r+1}+ \textstyle \frac{1}{\tau} \Vert n\tl - \n \Vert_{r-1}+ \Vert \dot{n}\tl - \np\Vert_{r-1}    \Big)\\
&+ \tau^2 K M_{t_{\ell+1},r+3}^p.
\end{aligned}
\end{equation}

\subsection{Error in $u$} Taking the difference of the approximation of  $u(t_{\ell})$ given in \eqref{Appu} and the numerical solution $u_{\ell}$ defined in \eqref{scheme}  we readily obtain by the definition of the remainder $\mathcal{R}_{r+2}$ (see Definition \ref{def:rem}) together with the relation
\begin{align*}
n\tl S^F\tl - \n \S &= n\tl S^F\tl - (\n-n\tl + n\tl) \S\\
& = (n\tl - \n) \S + n\tl (S^F\tl - \S)
\end{align*}
that
\begin{equation}\label{erru01}
\begin{aligned}
\Vert u(t_{\ell}) - u_{\ell}\Vert_{r+1} &\leq 
\left \Vert c^{-1} \nab^{-1}\left(\F - F\tl \right)\right\Vert_{r+1} 
\\&
 + \left\Vert  \nab^{-2}\left( (n\tl - \n)  S^F_\ell \right)\right\Vert_{r+1} + \left\Vert    \nab^{-2}\left( n\tl  \Big( S^F\tl - S^F_{\ell} \Big)\right)\right\Vert_{r+1} \\
& + \tau \mathcal{R}_{r+3}.
\end{aligned}
\end{equation}
Thanks to Lemma \ref{cnab} we have
\begin{align}\label{E3}
\left\Vert  \nab^{-2}\left( (n\tl - \n)  S^F_\ell \right)\right\Vert_{r+1} 
\leq  K \Vert  (n\tl - \n)  S^F_\ell \Vert_{r-1}
\leq K  \left(\frac{1}{\tau} \Vert n\tl - \n \Vert_{r-1} \right) \left(\tau \Vert S_\ell^F\Vert_{r+1}\right).
\end{align}
Furthermore, the bound \eqref{E2} with $u\tl - \u$ replaced by $(c\nab^{-1})(S^F\tl-S_\ell^F)$ implies that
\begin{equation}
\begin{aligned}\label{E4}
  \left\Vert   \nab^{-2}\left( n\tl  \Big( S^F\tl - S^F_{\ell} \Big)\right)\right\Vert_{r+1}
&  \leq  \left\Vert     \nab^{-2}\left( n\tl   (c \nab) \left[(c\nab)^{-1}\Big( S^F\tl - S^F_{\ell} \Big)\right]\right)\right\Vert_{r+1}
 \\&  \leq 
 K \Vert n\tl\Vert_{r+1} \left \Vert  (c\nab)^{-1} \Big( S^F\tl - S^F_{\ell} \Big)\right\Vert_{r+1}.
\end{aligned}
\end{equation}
Plugging \eqref{E3} and \eqref{E4} into \eqref{erru01} yields that
\begin{equation}\label{erru1}
\begin{aligned}
\Vert u(t_{\ell}) - u_{\ell}\Vert_{r+1} &\leq 
\left \Vert c^{-1} \nab^{-1}\left(\F - F\tl \right)\right\Vert_{r+1} 
\\&
+ K  \left(\frac{1}{\tau} \Vert n\tl - \n \Vert_{r-1} \right) \left(\tau \Vert S_\ell^F\Vert_{r+1}\right) \\
 &+  K \Vert n\tl\Vert_{r+1} \left \Vert  (c\nab)^{-1} \Big( S^F\tl - S^F_{\ell} \Big)\right\Vert_{r+1} \\
& + \tau \mathcal{R}_{r+3}.
\end{aligned}
\end{equation}

Taking the difference of $S^F\tl$ defined in \eqref{defSF} and $S^F_\ell$ given through \eqref{scheme}  we obtain that
\begin{equation}\label{Sappi0}
\begin{aligned}
 \left \Vert  (c\nab)^{-1} \Big( S^F\tl - S^F_{\ell} \Big)\right\Vert_{r+1} & = \tau \sum_{k=0}^\ell  \left \Vert (c\nab)^{-1} \varphi_1(i \tau c \nab)  \Big( F(t_k) - F_k\Big) \right \Vert_{r+1}\\
  &  \leq \tau \sum_{k=0}^\ell  \left \Vert (c\nab)^{-1}  \Big( F(t_k) - F_k\Big) \right \Vert_{r+1},
\end{aligned}
\end{equation}
where we have used that $\left \Vert \varphi_1(i \tau c \nab)\right\Vert_r \leq 1$. 

The definition of $S_\ell^F$ (see \eqref{scheme} with initial value \eqref{initial}) also yields that
\[
\Vert S_\ell^F\Vert_{r+1} \leq  \Vert u(0)\Vert_{r+1}+ \tau \sum_{k=0}^\ell  \left \Vert \varphi_1(i \tau c \nab) F_k \right \Vert_{r+1}.
\]
From the estimate \eqref{bphi} we can furthermore  conclude that
\begin{equation}
\begin{aligned}\label{Sappi}
 \tau \left \Vert  S^F_\ell \right \Vert_{r+1} &\leq \tau \Vert u(0)\Vert_{r+1}+\tau \sum_{k=0}^\ell  \left \Vert  \tau \varphi_1(i \tau c \nab)  F_k\right \Vert_{r+1}\\& 
 \leq \tau \Vert u(0)\Vert_{r+1}+ \tau \sum_{k=0}^\ell  \left \Vert  (c \nab)^{-1}  F_k\right \Vert_{r+1}\\
 & \leq \tau \Vert u(0)\Vert_{r+1}+ t_{\ell+1} \sup_{0 \leq k \leq \ell} \left \Vert  (c \nab)^{-1}  F_k \right \Vert_{r+1}.
\end{aligned}
\end{equation}
Plugging \eqref{Sappi0}  and \eqref{Sappi}  into \eqref{erru1} we thus obtain by the definition of $B_{t_\ell,r}$ in \eqref{Mn} that
\begin{equation}\label{erru}
\begin{aligned}
\Vert u(t_{\ell}) - u_{\ell}\Vert_{r+1} &\leq
\left \Vert c^{-1} \nab^{-1}\left(\F - F\tl \right)\right\Vert_{r+1} 
\\& +K  B_{t_\ell,r} t_\ell \left(\frac{1}{\tau} \Vert n\tl - \n \Vert_{r-1} \right) 
\\& +
K M_{t_\ell,r+1} \left(\tau\sum_{k=0}^{\ell}  \left \Vert c^{-1} \nab^{-1}\left(F(t_k) - F_k \right)\right\Vert_{r+1}\right) .
\end{aligned}
\end{equation}

\subsection{Error in $(n,\dot{n})$}
In the following we define the rotation matrix
\begin{equation}\label{rot}
\mathcal{D}(\tau \nabo) = 
\begin{pmatrix} \coso &  \sino \\ - \sino & \coso\end{pmatrix}.
\end{equation}
Taking the difference of the approximation to the exact solution $(n(t_{\ell+1}), \dot{n}(t_{\ell+1}))$ given in \eqref{Appnnp} and the numerical solution $(n_{\ell+1},\dot{n}_{\ell+1})$ defined in \eqref{scheme}  we readily obtain by the definition of the remainder $\mathcal{R}_{r+2}$ (see Definition \ref{def:rem}), the rotation matrix \eqref{rot} and the relation
\begin{equation}\label{uuu}
\begin{aligned}
& \vert u\tl  \vert - \vert \u^2 \vert 
 = (u\tl - \u)  \overline{u\tl }+ \u \overline{(u\tl - \u)},\\
 & u\tl^2 - \u^2 = (u\tl - \u) (u\tl+\u)
\end{aligned}
\end{equation}
(with the corresponding complex conjugate version) that 
\begin{equation*}
\begin{aligned}
 \begin{pmatrix} 
 n(t_{\ell+1}) - n_{\ell+1} \\\nabo^{-1}( \dot{n}\tl - \np)
 \end{pmatrix}
 &= \mathcal{D}(\tau\nabo) 
  \begin{pmatrix} 
n(t_{\ell}) - n_{\ell}  \\\nabo^{-1}( \dot{n}\tl - \np)
 \end{pmatrix}
\\& +\frac{ \tau}{4} \begin{pmatrix}
 \tau \frac{\sino}{\tau \nabo}  \nabo^2 \Big(p_1(u\tl,u_\ell) (u\tl - u_\ell) \Big)
  \\ \coso   \nabo \Big(p_2(u\tl,u_\ell) (u\tl - u_\ell) \Big)
 \end{pmatrix} 
 + 
 \begin{pmatrix}
 \tau^3  \mathrm{sinc}(\tau\nabo) \mathcal{R}_{r+4} \\  \tau^2 \mathcal{R}_{r+3}
 \end{pmatrix}.
\end{aligned}
\end{equation*}
Thereby,  $p_1$ and $p_2$ simply denote  polynomials in $u\tl, \u$ (according to \eqref{uuu}) due to the bounds  (see \eqref{def:phifunc})
\[
\left \vert \varphi_{1}(\pm 2 i c^2)\right \vert \leq 1, \qquad \left \vert \varphi_2(\pm 2 i c^2) \right \vert \leq 1.
\]

Solving the above recursion we obtain that
\begin{equation*}
\begin{aligned}
&  \begin{pmatrix} 
 n(t_{\ell+1}) - n_{\ell+1} \\\nabo^{-1}( \dot{n}\tl - \np)
 \end{pmatrix}
 \\&= \frac{ \tau}{4}  \sum_{k = 0}^\ell \mathcal{D}(\tau\nabo)^{k} 
 \begin{pmatrix}
 \tau \frac{\sino}{\tau \nabo}  \nabo ^2\Big(p_1(u(t_{n-k},u_{n-k}) (u(t_{n-k})- u_{n-k}) \Big)
  \\ \coso   \nabo \Big(p_2(u(t_{n-k},u_{n-k}) (u(t_{n-k})- u_{n-k})  \Big)
 \end{pmatrix} 
 \\& + \ell\tau
\begin{pmatrix}
 \tau^2\mathrm{sinc}(\tau\nabo)  \mathcal{R}_{r+4} \\\tau \mathcal{R}_{r+3}
 \end{pmatrix}
\\
 &= \frac{ \tau}{4}  \sum_{k = 1}^\ell \mathcal{D}(\tau\nabo)^{k-1} \left\{\mathcal{N}(u(t_{n-k}),u_{n-k}) \right\}
  \\&+ 
\frac{\tau}{4}
 \begin{pmatrix}
 \tau \frac{\sino}{\tau \nabo}  \nabo ^2\Big(p_1(u(t_{n},u_{n}) (u(t_{n})- u_{n}) \Big)
  \\ \coso   \nabo \Big(p_2(u(t_{n},u_{n}) (u(t_{n})- u_{n})  \Big)
 \end{pmatrix} 
 +  t_\ell\begin{pmatrix}
 \tau^2\mathrm{sinc}(\tau\nabo)  \mathcal{R}_{r+4} \\ \tau \mathcal{R}_{r+3}
 \end{pmatrix},
\end{aligned}
\end{equation*}
where we have set
\begin{align*}
\mathcal{N}(u(t_{n-k}),u_{n-k}) &= \begin{pmatrix} \mathcal{N}_1(u(t_{n-k}),u_{n-k}) \\ \mathcal{N}_2(u(t_{n-k}),u_{n-k})\end{pmatrix} \\
& :=
 \mathcal{D}(\tau\nabo) \begin{pmatrix}
 \tau \frac{\sino}{\tau \nabo}  \nabo ^2\Big(p_1(u(t_{n-k},u_{n-k}) (u(t_{n-k})- u_{n-k}) \Big)
  \\ \coso   \nabo \Big(p_2(u(t_{n-k},u_{n-k}) (u(t_{n-k})- u_{n-k})  \Big)
 \end{pmatrix} . 
\end{align*}

Note that for all $k \geq 1$ it holds that
\begin{equation}\label{Db}
\left \Vert \mathcal{D}(\tau \nabo)^{k-1} \right\Vert_r \leq 1 .
\end{equation}
Together with the observation
\begin{align*}
 \mathcal{N}_1(u(t_{n-k}),u_{n-k})
 & =   \coso  \tau \frac{\sino}{\tau \nabo}  \nabo ^2\Big(p_1(u(t_{n-k},u_{n-k}) (u(t_{n-k})- u_{n-k}) \Big)
\\& \quad + \sino \coso   \nabo \Big(p_2(u(t_{n-k},u_{n-k}) (u(t_{n-k})- u_{n-k})  \Big)\\
 & =  \tau   \coso\frac{\sino}{\tau \nabo}  \nabo ^2\Big\{\Big(p_1(u(t_{n-k},u_{n-k}) (u(t_{n-k})- u_{n-k}) \Big)\\
 &\qquad \qquad \qquad \qquad \qquad \qquad \qquad
+  \Big(p_2(u(t_{n-k},u_{n-k}) (u(t_{n-k})- u_{n-k})  \Big)\Big\}\\
\end{align*}
we thus obtain that
\begin{align}\label{NO}
&   \Vert  n(t_{\ell+1}) - n_{\ell+1} \Vert_{r-1} \leq  \tau  K M_{t_\ell,r} B_{t_\ell,r}
\left( \tau \sum_{k =0}^\ell \Vert u(t_k) - u_k\Vert_{r+1}\right)+\tau^2M_{t_{\ell+1},r+3}^p
\end{align}
as well as the (classical) bound
\begin{align*}
 \Vert  n(t_{\ell+1}) - n_{\ell+1} \Vert_{r} + \Vert \nabo^{-1}( \dot{n}\tl - \np)\Vert_{r-1}& \leq
K M_{t_\ell,r} B_{t_\ell,r}
\left( \tau \sum_{k =0}^\ell \Vert u(t_k) - u_k\Vert_{r+1}\right)\\
&+\tau M_{t_{\ell+1},r+3}^p.
\end{align*}
Hence, we can conclude that
\begin{equation}
\begin{aligned}\label{errnnp}
 \frac{1}{\tau}  \Vert  n(t_{\ell+1}) - n_{\ell+1} \Vert_{r-1}+
\Vert  n(t_{\ell+1}) - n_{\ell+1} \Vert_{r} + \Vert \dot{n}\tl - \np\Vert_{r-1} 
\\\leq   K M_{t_\ell,r} B_{t_\ell,r}
\left( \tau \sum_{k =0}^\ell \Vert u(t_k) - u_k\Vert_{r+1}\right)+\tau M_{t_{\ell+1},r+3}^p.
\end{aligned}
\end{equation}
\subsection{The convergence theorem}
The numerical solutions $(\u,\F,\n,\np)$ defined by the oscillatory  integration scheme \eqref{scheme} allows a first-order approximation to the exact solution \newline$(u\tl, F\tl,n\tl,\dot{n}\tl)$ of the Klein-Gordon-Zakharov system \eqref{kgzS} uniformly in $c$. More precisely, with
\[
z_\ell := \frac12 ( u_\ell + \overline{u}_\ell)
\]
(recall the transformation \eqref{z}) the following convergence result holds.
\begin{thm}[Convergence] Fix $r>d/2$. Assume that $(u(0),n(0),\dot{n}(0)) \in H^{r+4} \times H^{r+3} \times H^{r+2}$. Then there exist constants $T>0, \tau_0>0,K>0$ such that for all $t_\ell \leq T, \tau \leq \tau_0$  and all \textcolor{red}{$c\geq1$} we have that
\begin{equation}
\Vert z(t_\ell) - z_\ell \Vert_{r+1} + \Vert n\tl - \n\Vert_{r} + \Vert \dot{n}\tl - \np\Vert_{r-1} \leq K \tau,
\end{equation}
where the constant $K$ depends on T, on $M_{T,r+3}$ defined in \eqref{Mr}, and on $r$, but can be chosen independently of $c$.\label{thm:conv}
\end{thm}

\begin{proof} Due to the local wellposedness of the Klein-Gordon-Zakharov system \eqref{kgzS}  (see, e.g., \cite{LWP1a,LWP1b,LWP2}) we know that there exists a $T>0$ such that  $M_{T,r+3}$ defined in \eqref{Mr} is finite. Thereby, observe that by the definition of $F = \partial_t u$ we have (see \eqref{kgz})
\[
i F = - c \nab u  - \frac12 c \nab^{-1} n (u+\overline u)
\]
such that by Lemma \ref{cnab} and the bilinear estimate \eqref{bil2} we obtain
\[
\Vert (c \nab)^{-1} F \Vert_{r+1}\leq \Vert u \Vert_{r+1} +  \Vert n u \Vert_{r-1} \leq \Vert u \Vert_{r+1} + K\Vert n \Vert_{r-1} \Vert u \Vert_{r+1}.
\]

Collecting the error bounds \eqref{errF}, \eqref{erru} and \eqref{errnnp} yields that
\begin{equation}\label{errS}
\begin{aligned}
& \big\Vert (c\nab)^{-1} \left(F(t_\ell+\tau)-F_{\ell+1}\right)\big\Vert_{r+1} \\&\hskip2cm\leq \left \Vert (c\nab)^{-1} \left(F\tl-\F\right)\right\Vert_{r+1}\\
&  \hskip2cm+ \tau K  M_{t_l,r+1} \B \Big(\Vert u\tl - \u\Vert_{r+1}+ \textstyle \frac{1}{\tau} \Vert n\tl - \n \Vert_{r-1}+ \Vert \dot{n}\tl - \np\Vert_{r-1}    \Big)\\
&\hskip2cm + \tau^2 K M_{t_{\ell+1},r+3}^p,\\
& \Vert u(t_{\ell+1}) - u_{\ell+1}\Vert_{r+1}\\&\hskip2cm \leq
\left \Vert c^{-1} \nab^{-1}\left(F(t_{\ell+1}) - F_{\ell+1}\right)\right\Vert_{r+1} 
\\& \hskip2cm+K  B_{t_\ell,r} t_\ell \left(\frac{1}{\tau} \Vert n(t_{\ell+1}) - n_{\ell+1} \Vert_{r-1} \right) 
\\& \hskip2cm+
K M_{t_\ell,r+1} \left(\tau\sum_{k=0}^{\ell+1}  \left \Vert c^{-1} \nab^{-1}\left(F(t_k) - F_{k} \right)\right\Vert_{r+1}\right), \\
 &\frac{1}{\tau}  \Vert  n(t_{\ell+1}) - n_{\ell+1} \Vert_{r-1}+
\Vert  n(t_{\ell+1}) - n_{\ell+1} \Vert_{r} + \Vert \dot{n}(t_{\ell+1})- \dot{n}_{\ell+1}\Vert_{r-1} 
\\&\hskip2cm \leq   K M_{t_\ell,r} B_{t_\ell,r}
\left( \tau \sum_{k =0}^\ell \Vert u(t_k) - u_k\Vert_{r+1}\right)+\tau M_{t_{\ell+1},r+3}^p.
\end{aligned}
\end{equation}
In the following we assume that 
\[
\text{for all } k \leq \ell \quad : \quad B_{t_\ell,r} \leq M_1,\quad  M_{t_{\ell+1},r+1} \leq M_2, \quad
M_{t_{\ell+1},r+3} \leq M_3.
\]
Plugging the estimates on the error in $n$ and $\dot{n}$ into the error recursions in $u$ and $F$ yields together with
\[
\tau \sum_{k=0}^{\ell} \Vert f_k \Vert_r \leq  t_{\ell+1} \sup_{0 \leq k \leq \ell} \Vert f_k \Vert_r 
\]
by \eqref{errS}  that
\begin{equation}\label{rF}
\begin{aligned}
 \big\Vert (c\nab)^{-1} \left(F(t_\ell+\tau)-F_{\ell+1}\right)\big\Vert_{r+1} & \leq \left \Vert (c\nab)^{-1} \left(F\tl-\F\right)\right\Vert_{r+1}\\
& + \tau t_\ell \mathcal{K}_1(M_1,M_2,M_3) \sup_{0 \leq k \leq \ell}  \Vert u\tl - \u\Vert_{r+1}\\
& +  \tau^2 K M_{t_{\ell+1},r+3}^p
\end{aligned}
\end{equation}
as well as
\begin{equation}\label{ru}
\begin{aligned}
 \Vert u(t_{\ell+1}) - u_{\ell+1}\Vert_{r+1}&
  \leq t_\ell  \mathcal{K}_2(M_1,M_2,M_3) \sup_{ 0 \leq k \leq \ell} 
\left \Vert c^{-1} \nab^{-1}\left(F(t_{k+1}) - F_{k+1}\right)\right\Vert_{r+1} 
\\&+ \tau K B_{t_\ell,r} t_\ell  M_{t_{\ell+1},r+3}^p,
\end{aligned}
\end{equation}
where the constants $\mathcal{K}_1$ and $\mathcal{K}_2$depend on $t_\ell, M_1,M_2$ and $M_3$, but can be chosen independently of $c$.

Plugging \eqref{ru} into \eqref{rF} finally yields with the inductive assumption that the error in $F$ is growing that 
\begin{equation}\label{rF}
\begin{aligned}
 \big\Vert (c\nab)^{-1} \left(F(t_\ell+\tau)-F_{\ell+1}\right)\big\Vert_{r+1} &
 \leq \left(1+  \widetilde{\mathcal{K}}_1(t_\ell,M_1,M_2,M_3)  \tau \right) \left \Vert (c\nab)^{-1} \left(F\tl-\F\right)\right\Vert_{r+1} \\
 & + \tau^2 \widetilde{\mathcal{K}}_2(t_\ell,M_1,M_2,M_3) ,
 \end{aligned}
\end{equation}
where $\widetilde{\mathcal{K}}_1$ and $\widetilde{\mathcal{K}}_2$ depend on $t_\ell, M_1,M_2$ and $M_3$, but can be chosen independently of $c$.

Collecting the estimates in \eqref{rF}, \eqref{ru} and \eqref{errnnp} the assertion then follows by  the transformation \eqref{z} together with an  inductive, respectively,  \emph{Lady Windermere's fan} argument (see, for example~\cite{HNW93,Lubich08}).
\end{proof}
The uniform convergence rate in $c$ stated in Theorem \ref{thm:conv} is numerically confirmed in Figure \ref{fig1}.

\subsection{Asymptotic convergence}\label{AC} The oscillatory  integrator \eqref{scheme} is \emph{asymptotic preserving}  in the sense that  it converges asymptotically (i.e., for $c \to \infty$)  to the solution of the corresponding Zakharov limit system \eqref{zak} (for sufficiently smooth solutions). 

\begin{rem}[The Zakharov limit] \label{Z1}
Note that exact solutions $(z,n)$ of the Klein-Gordon-Zakharov system \eqref{eq:kgzOr} converge asymptotically  to the Zakharov system \eqref{zak} in the following sense: For sufficiently smooth solutions we have that (see, e.g., \cite{Berg96,LWP1a,LWP1b,Texier07})
\begin{equation}
\begin{aligned}\label{appc}
& z(t,x) = \frac12\left( \e^{i c^2 t} u^\infty(t,x) + \e^{-ic^2t} \overline{u}^\infty(t,x)\right) + c^{-2} \mathcal{R}_{r+4},\\
& n(t,x) = n^\infty(t,x) + c^{-2} \mathcal{R}_{r+5},
\end{aligned}
\end{equation}
where $(u^\infty,n^\infty)$ solve the Zakharov system (cf. \eqref{zak})
\begin{equation}\label{zaki}
\begin{aligned}
& 2 i \partial_t u^\infty  = \Delta u^\infty -  n^\infty u^\infty, \\
& \partial_{tt} n^\infty - \Delta n^\infty =\textstyle \frac12 \Delta \left \vert u^\infty\right\vert^2
\end{aligned}
\end{equation}
equipped with the initial values
\[
u^\infty(0) =z(0) - i c^{-2} \partial_t z(0) ,\quad n^\infty(0) = n(0) \quad \text{and}\quad \dot{n}^\infty(0)= \dot{n}(0).
\]

Formally, the Zakharov limit system \eqref{zaki} can be derived by introducing the \emph{twisted variable} (cf. e.g., \cite{LWP1a,LWP1b})
\begin{equation}\label{twist}
u_\ast(t) := \e^{-ic^2 t} u(t).
\end{equation}
The product rule together with equation \eqref{kgz} yields the following equation in $u_\ast$
\begin{equation}\label{uast}
\begin{aligned}
i \partial_t \ua &= c^2\ua + \e^{-ic^2 t} \partial_t u \\
& =  - (c\nab-c^2) \ua - \frac12 c\nab^{-1} n \left( \ua + \e^{-2ic^2t} \overline{u^\ast}\right).
\end{aligned}
\end{equation}
Together with the approximations (cf. Lemma \ref{cnab})
\[
c\nab - c^2 \rightarrow -\textstyle\frac12 \Delta+\mathcal{O}\left( c^{-2} \Delta^2\right) , \qquad c\nab^{-1} \rightarrow 1 + \mathcal{O}\left( c^{-2} \Delta\right)
\]
we formally obtain that
\begin{equation*}\label{uast}
\begin{aligned}
2 i \partial_t \ua  = \Delta \ua -  n\ua -  \e^{-2ic^2t} n \overline{u^\ast} + c^{-2} \mathcal{R}_{r+4}.
\end{aligned}
\end{equation*}
The equation in $n$ in terms of $\ua$ is given by (see \eqref{kgz})
\begin{equation*}
\partial_{tt} n - \Delta n = \frac14\Delta  \left(2
\vert \ua \vert^2 + \e^{2ic^2t}(\ua)^2 + \e^{-2ic^2t}(\overline{u^\ast})^2
\right ).
\end{equation*}
For a smooth function $f$ we furthermore have (by integration by parts) that
\begin{equation*}
\begin{aligned}
\int_0^t \e^{\pm 2 ic^2 \xi } f\left(\ua(\xi),n(\xi)\right) \dd \xi & =  \frac{1}{\pm 2 ic^2} 
\left( \e^{\pm 2 ic^2 t} f\left(\ua(t),n(t)\right) -  f\left(\ua(0),n(0)\right) \right)\\&
 + \frac{1}{\mp 2 ic^2} \int_0^t \e^{\pm 2 ic^2 \xi } \partial_\xi f\left(\ua(\xi),n(\xi)\right) \dd \xi.
\end{aligned}
\end{equation*}
 We can thus conclude the Zakharov limit system \eqref{zaki} (for sufficiently smooth solutions) thanks to the uniform boundedness of $\partial_t \ua$ and $\partial_t n$:
\[
\Vert \partial_\xi \ua(\xi)  \Vert_{r+1} + \Vert \partial_\xi n(\xi)\Vert_r \leq K M_{T,r+2}
\]
which hold thanks to Lemma \ref{cnab} and Lemma \ref{lem:appn} for some constant $K>0$ independent of $c$ and $M_{T,r}$ defined in \eqref{Mr}. 
\end{rem}
Theorem \ref{thm:conv} together with the approximation in \eqref{appc} implies that the  oscillatory  integrator \eqref{scheme} converges at order $\tau + c^{-2}$ towards the limit solutions of the Zakharov system. More precisely, for sufficiently smooth solutions the scheme \eqref{scheme} allows an approximation towards the solutions $(u^\infty,n^\infty)$ of the Zakharov limit system \eqref{zaki} with the convergence rate
\begin{equation*}
\Vert u^\infty (t_\ell) - \e^{-ic^2 t_\ell} \u \Vert_{r+1} + \Vert n^\infty \tl - \n\Vert_{r} + \Vert \dot{n}^\infty \tl - \np\Vert_{r-1} \leq K \left(\tau + c^{-2}\right)
\end{equation*}
for some constant $K>0$ which is independent of $\tau$ and $c$. With the  transformation \eqref{z} at hand we can in particular deduce for
\begin{align*}
  z^\infty := \frac12\left( \e^{i c^2 t} u^\infty(t,x) + \e^{-ic^2t} \overline{u}^\infty(t,x)\right)\quad \text{and}\quad z_\ell := \frac12( u_\ell + \overline u_\ell)
\end{align*}
that
\begin{equation}\label{b:asymp}
\left\Vert  z^\infty\tl - z_\ell \right\Vert_{r+1} + \Vert n^\infty \tl - \n\Vert_{r} + \Vert \dot{n}^\infty \tl - \np\Vert_{r-1} \leq K \left(\tau + c^{-2}\right).
\end{equation}

The asymptotic convergence \eqref{b:asymp} of our scheme \eqref{scheme} towards the solutions of the Zakharov limit system \eqref{zaki} is numerically confirmed in Figure \ref{fig3}.

\section{A second order uniformly accurate oscillatory integrator for KGZ}

Our novel technique allows us to develop  higher-order uniformly accurate scheme for the Klein-Gordon-Zakharov system \eqref{eq:kgzOr} which confirm  uniformly in $c$. In this section we develop  a \emph{second-order uniformly accurate integrator} which allows an error at order $\mathcal{O}(\tau^2$) uniformly in $c$. The construction of the second order scheme is based on iterating Duhamel's formula  \eqref{Duh}  in $(F,n, \dot{n})$  and integrating the occuring highly oscillatory phases $
\e^{i k c^2 t} $ $(k \in \mathbb{Z})$ and their interactions exactly. 
\subsection{Collection of essential lemma  and notation}
\begin{defn}\label{def:Ac}
In the following we define the operator
\[
\Ac = c \nab - c^2.
\]
\end{defn}

\begin{lem}\label{lem:2} Locally second-order uniform approximations to $(u,n,\dot{n})(t_\ell+\xi)$ are given by
\begin{align*}
u(t_\ell+\xi) & =  \e^{i \xi c \nab} u(t_\ell)    \red+ i \black\xi \frac12 c \nab^{-1}  \e^{i c^2 \xi }
 \Big\{ n(t_\ell) u(t_\ell) +\varphi_1\left(-i \xi (c\nab+ c^2)\right) \big(n(t_\ell)\overline u(t_\ell)\big)\Big\}\\
 &  + \xi^2 \mathcal{R}_{r+4},\\
n(t_\ell+\xi)  &= n(t_\ell) + \xi \dot{n}(t_\ell) 
 + \xi^2 \mathcal{R}_{r+4},\\
\dot{n}(t_\ell+\xi)  &= - \xi \nabo^2 n(t_\ell) +\dot{n}(t_\ell)  + \frac{\xi}{4} \Delta 
\left\{2 \vert u(t_\ell)\vert^2 + \varphi_1(2i c^2\xi) u(t_\ell)^2 +  \varphi_1(-2i c^2\xi) \overline{u(t_\ell)}^2
\right\} + \xi^2 \mathcal{R}_{r+4}.
\end{align*}
\end{lem}
\begin{proof}
Duhamel's formula in $u$ (see \eqref{DuhU}) together with the approximations (see Lemma \ref{cnab})
\[
u(t_\ell+\zeta) = \e^{i c^2 \zeta} u(t_\ell) + \zeta \mathcal{R}_{r+2},\qquad \e^{i \zeta (c^2 - c\nab)} = \e^{ i \zeta \frac12 \Delta} + \zeta \mathcal{R}_{r+4}
\]
 implies by  the definition of the $\varphi_1-$function (see \eqref{def:phifunc}) that
\begin{equation}
\begin{aligned}
u(t_\ell+\xi)  
 & =  \e^{i \xi c \nab}u(t_\ell)   \red+ \frac{i}{2} \black \xi  c \nab^{-1}  \e^{i \xi  c \nab}
 \Big(\textstyle \varphi_1\left( i \xi  \frac12\Delta \right) \big( n(t_\ell) u(t_\ell)\big)\\&+\varphi_1\left(-i \xi (c\nab+ c^2)\right) \big(n(t_\ell)\overline u(t_\ell)\big)\Big) + \xi^2 \mathcal{R}_{r+4}.
\end{aligned}
\end{equation}
Using the estimate in \eqref{appsin1} and the expansion $\varphi_1(i \xi \frac12 \Delta) = 1 + \mathcal{O}(\xi \Delta)$ implies  the assertion together with \eqref{Appnnp} by replacing $\tau $ with $\xi$ in the formulas for  $(n,\dot{n})$.
\end{proof}

\subsection{Construction of the second order scheme}

\subsubsection{Approximation in F}
Lemma \ref{lem:2}  together with the expansion $\varphi_1\left(-i \xi (c\nab+ c^2)\right) = 1 + \mathcal{O}(\xi c^2)$ implies (as $n(t)$ is real valued) that
 \begin{equation}\label{uF}
\begin{aligned}
 u(t_\ell+\xi) + \overline u(t_\ell + \xi)   & =\e^{i \xi c \nab}u(t_\ell) +\e^{-i \xi c \nab} \overline u (t_\ell)   + c^2 \xi^2 \mathcal{R}_{r+4}
 \end{aligned}
\end{equation}
Furthermore, we obtain with the aid of Lemma \ref{lem:2} that
\begin{equation}\label{nnp2}
\begin{aligned}
n(t_\ell+\xi)  &= n(t_\ell) + \xi \dot{n}(t_\ell) 
 + \xi^2 \mathcal{R}_{r+4},\\
\dot{n}(t_\ell+\xi)  
&=\textstyle  \dot{n}(t_\ell)  +\xi  \Delta \left(n(t_\ell) + \frac{1}{4} (u(t_\ell) + \overline u(t_\ell))^2 \right)+ c^2\xi^2 \mathcal{R}_{r+4}.
\end{aligned}
\end{equation}
Thanks to the relation \eqref{up} and approximation \eqref{AppF}  we furthermore obtain (as $n(t)$ is real-valued)  and $\Vert \nab^{-2} c \nab \Vert_r \leq 1$ that
\begin{equation}
\begin{aligned}\label{umu}
i c\nab & \left (u(t_\ell+\xi) - \overline u(t_\ell+\xi)) \right) \\&=  \left(F(t_\ell+\xi) + \overline F(t_\ell+\xi)\right)\\
& = \e^{i \xi c \nab} F(t_\ell) + \e^{-i \xi c \nab} \overline F(t_\ell)\\
& -  \frac{\xi}{2} c \nab^{-1} \big( n c \nab (\e^{i c^2 \xi} u(t_\ell) + \e^{-i c^2 \xi} \overline u(t_\ell)\big)\\
& +  \frac{\xi}{2} c \nab^{-1} \Big( \e^{i c^2 \xi} \varphi_1(-i \xi (c\nab + c^2)) n c \nab \overline u(t_\ell) + 
\e^{-i c^2 \xi} \varphi_1(i \xi (c\nab + c^2)) n c \nab  u(t_\ell) \Big)\\
& +c^2 \xi^2 \mathcal{R}_{r+4}.
\end{aligned}
\end{equation}
 Plugging the expansions \eqref{uF}, \eqref{nnp2} and \eqref{umu} into Duhamel's formula for $(c\nab)^{-1}F$ given in~\eqref{Duh} yields  as $\Vert c^2 \nab^{-2} \Vert_r \leq 1 $ that
 \begin{equation}
\begin{aligned}\label{2Duh}
& (c\nab)^{-1} F(t_\ell+\tau)  = \e^{i \tau c \nab}(c\nab)^{-1} F(t_\ell)  \red\\&+ \frac{i}{2}\black   \e^{i \tau c \nab} \nab^{-2} \int_0^\tau  \e^{- i \xi c \nab}
\Big\{ \Big(\dot{n}(t_\ell)  +\xi  \Delta \big(n(t_\ell) + \frac{1}{4} (u(t_\ell) + \overline u(t_\ell))^2 \big)  \Big)\left(\e^{i \xi c \nab}u(t_\ell) +\e^{-i \xi c \nab} \overline u (t_\ell)  \right) \\
 & +     \Big(n(t_\ell) + \xi \dot{n}(t_\ell) \Big)  \big(F(t_\ell+\xi)+\overline F(t_\ell+\xi)\big) \Big\}\dd \xi + \tau^3 \mathcal{R}_{r+4} .
\end{aligned}
\end{equation}
In the following we set
\begin{align}\label{II}
\mathcal{I}_0(u(t_\ell), n(t_\ell)) :=  u(t_\ell) \Delta \big(n(t_\ell) + \frac14 (u(t_\ell) + \overline u(t_\ell))^2 \big)+ i \dot{n}(t_\ell) \Ac u(t_\ell).
\end{align}
With the aid of the expansion $\e^{ i \xi c\nab} = \e^{i \xi c^2} (1 + i \xi \Ac) + \mathcal{O}(\xi^2 \Ac^2)$ we then obtain together with Definition \ref{Def:phifunc}  that
\begin{equation}\label{a}
\begin{aligned}
&\int_0^\tau   \e^{- i \xi c \nab}\Big(\dot{n}(t_\ell)  +\xi  \Delta \big(n(t_\ell) + \frac{1}{4} (u(t_\ell) + \overline u(t_\ell))^2 \big)  \Big)\left(\e^{i \xi c \nab}u(t_\ell) +\e^{-i \xi c \nab} \overline u (t_\ell)  \right)\dd \xi \\
& = \int_0^\tau   \e^{- i \xi (c \nab-c^2)} \dot{n}(t_\ell)  (1+ i \xi \Ac) u(t_\ell) +\e^{-i \xi (c \nab+c^2) } \dot{n}(t_\ell)  (1 - i \xi \Ac) \overline u (t_\ell)  \dd \xi\\
& + \int_0^\tau   \xi \e^{-i \xi \nab}  \big( \e^{i \xi c^2} u(t_\ell) + \e^{-i \xi^2} \overline u(t_\ell)\big) \Delta (n(t_\ell) + \frac14 (u(t_\ell)+\overline u (t_\ell))^2)  \dd \xi\\
& = \tau \varphi_1(- i \tau \Ac) \dot{n}(t_\ell) u(t_\ell) +\tau  \varphi_1(- i \tau (c \nab+c^2)) \dot{n}(t_\ell) \overline u(t_\ell) 
\\&+ \tau^2 \Psi_2(- i \tau\Ac )\mathcal{I}_0(u(t_\ell),n(t_\ell))  + \tau^2  \Psi_2(- i \tau (c\nab+c^2)) \overline{\mathcal{I}_0(u(t_\ell),n(t_\ell))}.
\end{aligned}
\end{equation}
Using \eqref{uuu} we similarly  obtain that
\begin{equation}\label{b}
\begin{aligned}
&  \int_0^\tau \e^{- i \xi c \nab} \left(\Big(n(t_\ell) + \xi \dot{n}(t_\ell) \Big)  \big(F(t_\ell+\xi)+\overline F(t_\ell+\xi)\big) \right)
\dd \xi \\&
=    \int_0^\tau  \Big\{ \e^{- i \xi \Ac } n(t_\ell) (1+ i \xi \Ac) F(t_\ell) + \e^{-i \xi (c \nab+c^2)} n(t_\ell)  (1- i \xi \Ac) \overline F(t_\ell)\\
& - \frac{\xi}{2} n(t_\ell)  c \nab^{-1} \big( \e^{- i  \xi\Ac }   n c \nab u(t_\ell) + \e^{-i (c\nab+ c^2) \xi}    n c \nab \overline u(t_\ell)\big)\\
& +  \frac{\xi}{2} n(t_\ell) c \nab^{-1} \Big(  \varphi_1(-i \xi (c\nab + c^2)) n c \nab \overline u(t_\ell) + 
\e^{-2 i c^2 \xi} \varphi_1(i \xi (c\nab + c^2)) n c \nab  u(t_\ell) \Big)\Big\}
\dd \xi \\
& +  \int_0^\tau\Big\{ \xi \e^{ - i \xi \Ac} \big(\dot n(t_\ell) F(t_\ell) \big)+ \xi \e^{ - i \xi (c\nab+c^2)} \big(\dot n(t_\ell) \overline F(t_\ell) \big)\Big\}\dd \xi \\
& =  \tau \varphi_1(- i \tau \Ac ) n(t_\ell)  F(t_\ell)  +   \tau \varphi_1(-i \tau (c \nab+c^2)) n(t_\ell)   \overline F(t_\ell)   
\\& +   \tau^2 \Psi_2(- i \tau \Ac) \Big(i n(t_\ell) \Ac F(t_\ell)  + \dot n(t_\ell) F(t_\ell)\Big)
 +    \tau^2 \Psi_2(-i (c \nab + c^2)) \Big( - i  n(t_\ell)\Ac  \overline F(t_\ell) +  \dot n(t_\ell) \overline F(t_\ell)\Big)\\
& +   \frac{\tau^2}{2} n(t_\ell)  c \nab^{-1} 
\Big\{ \Big(-  \Psi_2(- i  \tau\Ac )  + \frac{\varphi_1(i \tau \Ac) - \varphi_1(-2 i c^2 \tau)}{ i\tau  (c \nab+c^2)}\Big)  n c \nab u(t_\ell) \\&
\qquad \qquad  + \Big( - \black \Psi_2(-i (c\nab+ c^2) \tau)  +  \varphi_2(-i \tau (c\nab + c^2))\Big)  n c \nab \overline u(t_\ell)\Big\}.
\end{aligned}
\end{equation}
Now we set 
\begin{align}\label{2II}
\mathcal{I}(u(t_\ell), n(t_\ell)) :=  u(t_\ell) \Delta \big(n(t_\ell) + \frac14 (u(t_\ell) + \overline u(t_\ell))^2 \big)+  \dot{n}(t_\ell) \big(i\Ac u(t_\ell) + F(t_\ell)\big) + i n(t_\ell) \Ac F(t_\ell) .
\end{align}
Plugging the calculations \eqref{a} and \eqref{b} into \eqref{2Duh} yields together with \eqref{II} that
\begin{equation}
\begin{aligned}\label{F2order}
(c\nab&)^{-1} F(t_\ell+\tau)  = \e^{i \tau c \nab}(c\nab)^{-1} F(t_\ell)  + \tau \frac{i}{2}\black   \e^{i \tau c \nab} \nab^{-2 }\\&\Big\{
\varphi_1(- i \tau \Ac) \Big( \dot{n}(t_\ell) u(t_\ell) +  n(t_\ell) F(t_\ell)\Big) + \varphi_1(- i \tau (c \nab+c^2)) \Big(\dot{n}(t_\ell) \overline u(t_\ell) +   n(t_\ell)   \overline F(t_\ell)  \Big)
\\&+ \tau \Psi_2(- i \tau\Ac )\mathcal{I}(u(t_\ell),n(t_\ell))   +  \tau  \Psi_2(- i \tau (c\nab+c^2)) \overline{\mathcal{I}(u(t_\ell),n(t_\ell))}\\&
+   \frac{\tau}{2} n(t_\ell)  c \nab^{-1} 
\Big[ \Big( -\Psi_2(- i  \tau\Ac )  + \frac{\varphi_1(i \tau \Ac) - \varphi_1(-2 i c^2 \tau)}{ i\tau  (c \nab+c^2)}\Big)  n c \nab u(t_\ell) \\&
\qquad \qquad  + \Big(  - \black \Psi_2(-i (c\nab+ c^2) \tau)  +  \varphi_2(-i \tau (c\nab + c^2))\Big)  n c \nab \overline u(t_\ell)\Big]\Big\}+ \tau^3 \mathcal{R}_{r+4}
\end{aligned}
\end{equation}
with $\mathcal{I}(u(t_\ell),n(t_\ell))$ defined in \eqref{2II}.
\subsubsection{Approximation in $u$} 
We have that (see \eqref{kgzS})
\begin{equation}\label{up2}
u(t_\ell) = (c \nab)^{-1}\left\{ - i F(t_\ell)   - \frac12 c\nab^{-1} n(t_\ell) \left( I^F(t_\ell) + \overline{I^F(t_\ell)}\right) \right\}
\end{equation}
with
\[
I^F(t_\ell) := u(0) + \int_0^{t_\ell} F(\xi)\dd\xi =  u(0) + \sum_{k=0}^{\ell-1} \int_0^\tau F(t_k+\xi) \dd \xi .
\]
Note that the expansion \eqref{AppF} implies that
\begin{equation}
\begin{aligned}\label{appIF2}
I^F(t_\ell) & = u(0) + \sum_{k=0}^{\ell -1} \int_0^\tau \Big\{ \e^{i \xi c \nab} F(t_k)
 + i \black \frac{\xi}{2}  c\nab^{-1}\e^{i c^2 \xi} \varphi_1\big(\textstyle i  \xi \frac12 \Delta\big) \Big(\dot{n}(t_k)  u(t_k) + i n(t_k)  c \nab u(t_k)\Big)\\
  &     \red\qquad \qquad\qquad  + i \black\frac{\xi}{2} c \nab^{-1}\e^{i c^2 \xi}\varphi_1\big(-i \xi( c\nab+c^2)\big)\Big(\dot{n}(t_k) \overline u(t_k)- i n(t_k)  c \nab \overline u(t_k)\Big) \Big\} \dd \xi +  \tau^2 \mathcal{R}_{r+4}.
\end{aligned}
\end{equation}
Setting
\begin{equation}\label{SF2order}
\begin{aligned}
S^{F,2}(t_\ell) & := \sum_{k=0}^{\ell -1}\Big\{\textstyle \tau \varphi_1(i \tau c \nab) F(t_k)
+ i \frac{\tau^2}{2} c\nab^{-1} \Psi_2(i c^2 \tau)\varphi_1(i \tau \frac12 \Delta) \black\Big(\dot{n}(t_k)  u(t_k) + i n(t_k)  c \nab u(t_k)\Big)\\
  &    +i\frac{\tau}{2}  c\nab^{-1}
  \frac{ i \tau c^2}{-i (c\nab+c^2)} \frac{\varphi_1(-i \tau c \nab) - \varphi_1(i \tau c^2)}{i \tau c^2}
\black\Big(\dot{n}(t_k) \overline u(t_k)- i n(t_k)  c \nab \overline u(t_k)\Big) \Big\}
  \end{aligned}
\end{equation}
yields together with Definition \ref{Def:phifunc} that 
\begin{equation}\label{u2order}
u(t_\ell) = (c \nab)^{-1}\left\{ - i F(t_\ell)   - \frac12 c\nab^{-1} n(t_\ell) \left(u(0)+ S^{F,2}(t_\ell) + \overline{u(0)+S^{F,2}(t_\ell)}\right) \right\} + \tau^2 \mathcal{R}_{r+4}.
\end{equation}

\subsubsection{Approximation in $(n,\dot{n})$}
In the following we set
\[
\mathcal{J}(u(t_\ell),n(t_\ell),\xi) :=  \frac12 c \nab^{-1} 
 \Big\{ n(t_\ell) u(t_\ell) +\varphi_1\left(-i \xi (c\nab+ c^2)\right) \big(n(t_\ell)\overline u(t_\ell)\big)\Big\}.
\] 
Plugging the expansion given in Lemma \ref{lem:2} into Duhamel's formula \eqref{Duh} we obtain with the aid of \eqref{appsin1} and the observation 
\begin{align*}
\left \vert \e^{i\xi  c \nab} u(t_\ell)\right\vert^2 
    & =\Big(( \e^{i \xi \Ac } - 1  ) u(t_\ell)\Big) \overline u (t_\ell) + u(t_\ell) \e^{ - i \xi \Ac} \overline u(t_\ell)
    + \mathcal{O}(\xi^2 \Ac u(t_\ell))\\
     \left( \e^{i\xi  c \nab} u(t_\ell)\right)^2
     & = - \e^{2i c^2 \xi} u^2(t_\ell) + 2 u(t_\ell) \e^{i \xi(c^2+c\nab)} u(t_\ell)+ \mathcal{O}(\xi^2 \Ac u(t_\ell))
\end{align*}
that
\begin{equation*}
\begin{aligned}
\dot{n}(t_\ell+\tau)  
 &= - \nabo \sino n(t_\ell) + \coso \dot{n}(t_\ell) \\
& + \frac14  \Delta \int_0^\tau \Big \{2
\overline u (t_\ell) (\e^{i \xi \Ac} - 1)u(t_\ell) + 2u(t_\ell) \e^{-i \xi \Ac} \overline u(t_\ell) \\
& - \e^{2ic^2 \xi} u^2(t_\ell) + 2 u(t_\ell) \e^{i \xi(c^2+c\nab)} u(t_\ell)  - \e^{-2ic^2 \xi} \overline u^2(t_\ell) + 2 \overline u(t_\ell) \e^{-i \xi(c^2+c\nab)} \overline u(t_\ell)
\\&
+  2i \xi \Big(\e^{2ic^2 \xi} u(t_\ell) \mathcal{J}(u(t_\ell),n(t_\ell),\xi) - \e^{- 2i c^2 \xi} \overline u(t_\ell) \overline{\mathcal{J}(u(t_\ell),n(t_\ell),\xi)}
\Big)\\&
+ 2i \xi \Big( -u(t_\ell) \overline{\mathcal{J}(u(t_\ell),n(t_\ell),\xi)}+ \overline u(t_\ell) \mathcal{J}(u(t_\ell),n(t_\ell),\xi)\Big) \Big\}\mathrm{d}\xi
\\ &+ \tau^3 \mathcal{R}_{r+5}.\\
\end{aligned}
\end{equation*}
Note that 
\begin{equation}\label{I12}
\begin{aligned}
2i & \int_0^\tau \xi \e^{2ic^2 \xi} u(t_\ell) \mathcal{J}(u(t_\ell),n(t_\ell),\xi) \dd \xi \\
 & = i \int_0^\tau \xi \e^{2ic^2 \xi} u(t_\ell) c \nab^{-1}  (n(t_\ell) u(t_\ell)) \dd \xi
 +  i \int_0^\tau \xi \e^{2ic^2 \xi} u(t_\ell) c \nab^{-1} \varphi_1\left(-i \xi (c\nab+ c^2)\right) \big(n(t_\ell)\overline u(t_\ell)\big)\dd \xi\\
& =   i  \tau^2\black  \mathrm{sinc}(\tau \nabo)\Psi_2(2ic^2\tau) u(t_\ell)  c\nab^{-1}(n(t_\ell) u(t_\ell) )
\\& \quad + i \black \tau u(t_\ell)  \frac{1}{  - i \black ( \nab + c)} \Big( \varphi_1(-i \tau \Ac) - \varphi_1(2ic^2\tau)\Big) \nab^{-1} (n(t_\ell) \overline u(t_\ell)) =:\tau  \mathcal{J}_1^{\dot n}(u(t_\ell),n(t_\ell),\tau )\\
2i &  \int_0^\tau \xi \overline u(t_\ell)  \mathcal{J}(u(t_\ell),n(t_\ell),\xi) \dd \xi \\
& =  i \mathrm{sinc}(\tau \nabo) \int_0^\tau \xi   \overline u(t_\ell) c \nab^{-1}  (n(t_\ell) u(t_\ell)) \dd \xi
\\
& +  i\mathrm{sinc}(\tau \nabo)\black \int_0^\tau \xi  \overline u(t_\ell) c \nab^{-1} \varphi_1\left(-i \xi (c\nab+ c^2)\right) \big(n(t_\ell)\overline u(t_\ell)\big) \dd \xi+ \tau^3 \mathcal{R}_{r+3}\\
 & =  i \frac{\tau^2}{2}\mathrm{sinc}(\tau \nabo)\black   \overline u(t_\ell) c \nab^{-1}  (n(t_\ell) u(t_\ell))
 \\&+  i \tau^2 \mathrm{sinc}(\tau \nabo)\black  \Big(\overline u(t_\ell)  \varphi_2\left(-i \tau (c\nab+ c^2)\right)c \nab^{-1} \big(n(t_\ell)\overline u(t_\ell)\big)\Big)+ \tau^3 \mathcal{R}_{r+3} \\
 & =:\tau \mathcal{J}_2^{\dot n}(u(t_\ell),n(t_\ell),\tau ).
\end{aligned}
\end{equation}
Together with Definition \ref{Def:phifunc} we thus obtain that
\begin{equation}\label{nnp2order}
\begin{aligned}
\dot{n}(t_\ell+\tau)  
 &= - \nabo \sino n(t_\ell) + \coso \dot{n}(t_\ell) \\
& + \frac14 \tau  \Delta\Big \{
2 \overline u (t_\ell) (\varphi_1(i \tau \Ac) - 1)u(t_\ell) +2  u(t_\ell) \varphi_1(-i \tau \Ac) \overline u(t_\ell)  - \varphi_1(2ic^2 \xi) u^2(t_\ell) \\&+ 2  u(t_\ell) \varphi_1(i \xi(c^2+c\nab)) u(t_\ell)  -\varphi_1(-2ic^2 \tau) \overline u^2(t_\ell) + 2\overline u(t_\ell) \varphi_1(-i \tau(c^2+c\nab))\overline u(t_\ell)
\\&
+  \mathcal{J}_1^{\dot n}(u(t_\ell),n(t_\ell),\tau) + \black   \overline{\mathcal{J}_1^{\dot n}}(u(t_\ell),n(t_\ell),\tau)+  \mathcal{J}_2^{\dot n}(u(t_\ell),n(t_\ell),\tau)+  \overline {\mathcal{J}_2^{\dot n}}(u(t_\ell),n(t_\ell),\tau) \Big\}
\\ &+ \tau^3 \mathcal{R}_{r+5}.
\end{aligned}
\end{equation}
\subsubsection{A uniformly accurate oscillatory integrator of second order}
Collecting the results \eqref{F2order} and \eqref{u2order} (together with \eqref{SF2order}) as well as the approximation \eqref{Appnnp} for $n$ and \eqref{nnp2order}  for $\dot n$ motivates us to  define our second order numerical scheme as follows: 

For $0 \leq \ell \leq n-1$ we set
\begin{equation}\label{scheme2ord}
\begin{aligned}
 \big(c \nab\big)^{-1}   F_{\ell+1} &  = \e^{i \tau c \nab}(c\nab)^{-1} F_\ell  + \tau \frac{i}{2}\black   \e^{i \tau c \nab} \nab^{-2 }\\&\Big\{
\varphi_1(- i \tau \Ac) \Big( \dot{n}_\ell u_\ell +  (\psi  n_\ell \black) F_\ell\Big) + \varphi_1(- i \tau (c \nab+c^2)) \Big(\dot{n}_\ell \overline u_\ell +   (\psi  n_\ell \black)  \overline F_\ell  \Big)
\\&+ \tau \psi  \black \Psi_2(- i \tau\Ac )\mathcal{\mathcal{I}}(u_\ell,n_\ell)   +  \tau  \Psi_2(- i \tau (c\nab+c^2)) \overline{\mathcal{\mathcal{I}}(u_\ell,n_\ell)}\\&
+   \frac{\tau}{2} \textstyle  \psi \black \Big(n_\ell  c \nab^{-1} 
\Big[ \Big( \Psi_2(- i  \tau\Ac )  + \frac{\varphi_1(i \tau \Ac) - \varphi_1(-2 i c^2 \tau)}{ i\tau  (c \nab+c^2)}\Big)  n c \nab u_\ell \\&
\qquad \qquad + \Big( - \black \Psi_2(-i (c\nab+ c^2) \tau)  +  \varphi_2(-i \tau (c\nab + c^2))\Big)  n c \nab \overline u_\ell\Big]\Big)\Big\}\\
  n_{\ell+1}  &= \coso n_\ell+ \nabo^{-1} \sino \dot{n}_\ell \\
& + \frac{\tau^2}{4} \sinco \Delta 
\left\{  \vert u_\ell \vert^2 + \varphi_2(2i c^2\tau) u_\ell ^2 + \varphi_2(-2ic^2\tau) \overline{u_\ell }^2
\right\}\\
\dot{n}_{\ell+1} 
 &= - \nabo \sino n_\ell + \coso \dot{n}_\ell \\
& + \frac14 \tau  \Delta\Big \{2
\overline u _\ell (\varphi_1(i \tau \Ac) - 1)u_\ell +2 u_\ell \varphi_1(-i \tau \Ac) \overline u_\ell  - \varphi_1(2ic^2 \xi) u^2_\ell \\&+ 2  u_\ell \varphi_1(i \xi(c^2+c\nab)) u_\ell  -\varphi_1(-2ic^2 \tau) \overline u^2_\ell + 2\overline u_\ell \varphi_1(-i \tau(c^2+c\nab))\overline u_\ell
\\&
 +  \mathcal{J}_1^{\dot n}(u_\ell,n_\ell,\tau)   + \overline{\mathcal{J}_1^{\dot n}(u_\ell,n_\ell,\tau)}    +\mathcal{J}_2^{\dot n}(u_\ell,n_\ell,\tau)+ \overline{\mathcal{J}_2^{\dot n}(u_\ell,n_\ell,\tau)} \Big\}\\
 S^{F,2}_{\ell+1}  & =  S^{F,2}_\ell + \tau \varphi_1(i \tau c \nab) F_\ell \\&
 + i \black \frac{\tau^2}{2}  \nab^{-2} \Psi_2(i\tau c^2) \varphi_1(i \textstyle \tau \frac12  \Delta)\Big(\dot{n}_\ell  u_\ell + i n_\ell  c \nab u_\ell\Big)\\
  &    +i\frac{\tau^2}{2}  c\nab^{-1}
  \frac{  c^2}{- (c\nab+c^2)} \frac{\varphi_1(-i \tau c \nab) - \varphi_1(i \tau c^2)}{i \tau c^2}
\black\Big(\dot{n}_\ell \overline u_\ell- i n_\ell  c \nab \overline u_\ell\Big)
\\ u_{\ell+1} &= (c \nab)^{-1}\left\{ - i F_{\ell+1}   - \frac12 c\nab^{-1} n_{\ell+1} \left( u_0 + S^{F,2}_{\ell+1} + \overline{u_0+S^{F,2}_{\ell+1}}\right) \right\} 
\end{aligned}
\end{equation}
with $\mathcal{I}(u_\ell,n_\ell)$ defined in \eqref{2II}, $\mathcal{J}^{\dot n}_1$ and $\mathcal{J}^{\dot n}_2$ given in \eqref{I12}, the initial values
\begin{equation}\label{initial}
\begin{aligned}
& u_0 := u(0), \qquad n_0:=n(0),\qquad \dot{n}_0 := \partial_t n(0),\\
& F_0 := i c \nab u_0      + \frac{i}{2}  \black c\nab^{-1} n_0 (u_0+\overline{u_0}), \qquad  S_0^{F,2} := 0
\end{aligned}
\end{equation}
and where we have included the filter function $ \psi = \textstyle\mathrm{sinc}(\tau\frac12 \Delta)$.

\subsection{Convergence of the second order scheme}
The numerical solutions $(\u,\F,\n,\np)$ defined by the oscillatory  integration scheme \eqref{scheme2ord} allow a second-order approximation to the exact solution $(u\tl, F\tl,n\tl,\dot{n}\tl)$ of the Klein-Gordon-Zakharov system \eqref{kgzS} uniformly in $c$. More precisely, with
\[
z_\ell := \frac12 ( u_\ell + \overline{u}_\ell)
\]
(recall the transformation \eqref{z}) the following convergence result holds.
\begin{thm}[Second order convergence] Fix $r>d/2$. Assume that $(u(0),n(0),\dot{n}(0)) \in H^{r+6} \times H^{r+5} \times H^{r+4}$. Then there exist constants $T>0, \tau_0>0,K>0$ such that for all $t_\ell \leq T, \tau \leq \tau_0$  and all \textcolor{red}{$c\geq1$} we have that
\begin{equation}
\Vert z(t_\ell) - z_\ell \Vert_{r+1} + \Vert n\tl - \n\Vert_{r} + \Vert \dot{n}\tl - \np\Vert_{r-1} \leq K \tau^2,
\end{equation}
where the constant $K$ depends on T, on $M_{T,r+5}$ defined in \eqref{Mr}, and on $r$, but can be chosen independently of $c$.\label{thm:conv2}
\end{thm}

\begin{proof}
The convergence analysis of the second order scheme \eqref{scheme2ord} follows the line of argumentation taken in Section \ref{sec:con1} in the analysis of the first order scheme  (see also the proof of Theorem~\ref{thm:conv}). Therefore we only give the central steps. \\
\emph{Local error:} The local error bound at order $\mathcal{O}(\tau^3)$ uniformly in $c$ under the stated regularity assumptions  follows by construction (e.g., taking the difference of the exact solution \eqref{F2order}, \eqref{u2order},  \eqref{nnp2order} and the numerical scheme  \eqref{scheme2ord}) together with the observation
 \[
 \left \Vert (\psi -1) f \right\Vert_r = \left \Vert\left( \textstyle\mathrm{sinc}(\tau\frac12 \Delta)-1\right)f \right \Vert_r \leq \textstyle \frac14 \tau^2 \Vert f \Vert_{r+4}
 \] 
 which implies that $\tau \Vert (\psi - 1) n(t_\ell) \Vert_{r+1} \leq \tau^3 \Vert n(t_\ell)\Vert_{r+5}$.\\
 \emph{Stability:} The stability analysis follows the line of argumentation given in Section \ref{sec:con1} for the first order scheme. The essential estimates
  \begin{align*}
 &\Vert \nab^{-2} (f c \nab g )\Vert_{r-1} \leq K \Vert f \Vert_{r-1} \Vert g \Vert_{r+1}\\
 &\tau \Vert \Phi(i \textstyle \tau \frac12 \Delta) f \Vert_{r+1} 
 + \tau \Vert \Phi(- i \tau(c\nab+c^2)) f \Vert_{r+1} \leq 2 \Vert f \Vert_{r-2} \quad \text{for} \quad \Phi = \psi, \varphi_1, \varphi_2, \Psi_2
 \end{align*}
 which follow from Lemma \ref{cnab} and Definition \ref{Def:phifunc} yield the estimate \eqref{errS} with the local error increased by one order. The assertion then follows by a \emph{Lady Windermere's fan} argument (see, for example~\cite{HNW93,Lubich08}), see also the proof of Theorem \ref{thm:conv}.
\end{proof}
\begin{rem} The second-order scheme \eqref{scheme2ord} asymptotically converges to the Zakharov limit stystem \eqref{zak} (see also Remark \ref{Z1}).

\end{rem}

\section{Numerical experiments}
In this section we numerically underline the theoretical convergence results presented in Theorem \ref{thm:conv} \Red and \ref{thm:conv2}\black. The numerical experiments in particular confirm the uniform convergence property in the high-plasma frequency~$c$. In the numerical experiments (see Figure \ref{fig1}) we observe that the error does not increase for $c \to \infty$ which is the main objective of the derived oscillatory integrator \eqref{scheme} and \eqref{scheme2ord}. Furthermore, we numerically underline the asymptotic convergence (see Section \ref{AC}) at order $c^{-2}$ of the scheme \eqref{scheme} towards solutions of the Zakharov limit system \eqref{zaki} (see Figure \ref{fig3}).

In the numerical experiments we use a standard Fourier pseudo spectral spatial discretization with the largest Fourier mode $K = 2^{6}$ which corresponds to a spatial mesh size $\Delta x =  0.19$. Furthermore, we  choose the ($c-$dependent)  initial values
\begin{equation*}
\begin{split}
z(0,x) = \frac{1}{4} \frac{\sin(x)}{2- \cos (2x)}, \quad \partial_t z(0,x) = c^2  z(0,x), \quad
n(0,x) =\frac{ \sin(x)\cos(x)}{2-\sin(2x)}, \quad \partial_t n(0,x) = \frac12 \sin(x)
\end{split}
\end{equation*}
\black
which we integrate up to $T = 1$ \black. In Figure \ref{fig1} we plot the time-step size versus the error of the oscillatory integration schemes of first- \eqref{scheme} and second-order \eqref{scheme2ord} for different values of $c$. The error in $z$ and $n$ is thereby measured in a discrete $H^1$ and $L^2$ norm, respectively. As a reference solution we use the scheme itself with a finer time step size $\tau \approx 6\cdot 10^{-4} $ \black. \\

%

\begin{figure}[h!]
	\centering
	\resizebox{.85\linewidth}{!}{\input{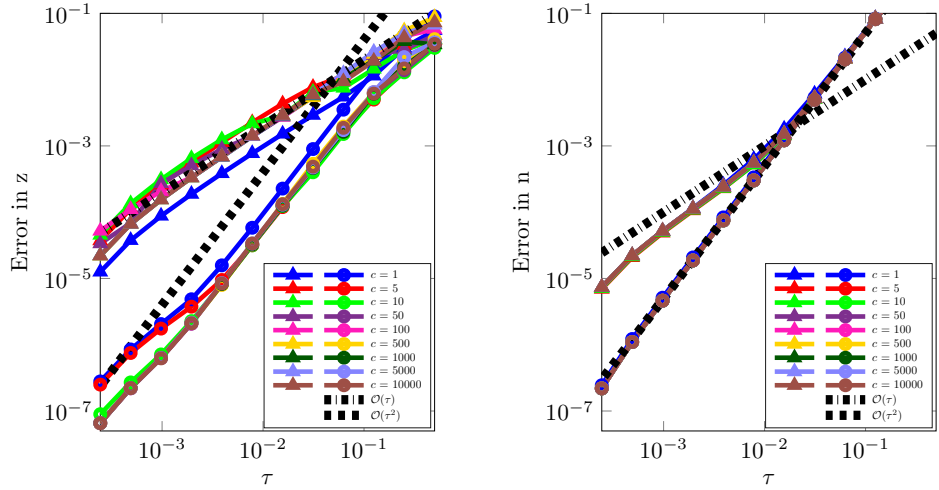}}
	\caption{Order plot of the oscillatory integration scheme of first-order \eqref{scheme} and second-order \eqref{scheme2ord} (double logarithmic scale): Error versus time step size. The slope of the dashed-dotted line is one and of the dashed line is two.}
	\label{fig1}
\end{figure}

%

In Figure  \ref{fig3} we numerically confirm the theoretical  convergence rate at order $c^{-2}$ given in \eqref{b:asymp} of the oscillatory integration scheme \eqref{scheme}  towards solutions of the Zakharov limit system \eqref{zaki}  (see also Section \ref{AC}). Thereby, we use the trigonometetric integrator  derived in \cite{HerrS17} to solve the Zakharov limit system~\eqref{zaki} numerically. Note that in \cite{HerrS17} the complex conjugated version of \eqref{zak} is considered (i.e., the equation in $\overline z$) with a scaling factor $\frac12$. The simulation is carried out with step size $\tau \approx 10^{-6}$. \\

\begin{figure}[h!]
	\centering
	\resizebox{.5\linewidth}{!}{
%
%
\definecolor{mycolor1}{rgb}{0.00000,0.44700,0.74100}%
\definecolor{mycolor2}{rgb}{0.85000,0.32500,0.09800}%
\definecolor{mycolor3}{rgb}{0.92900,0.69400,0.12500}%
\begin{tikzpicture}

\begin{axis}[%
width=4.5in,
height=2.0in,
at={(0.0in,0.0in)},
scale only axis,
xmode=log,
xmin=1,
xmax=1024,
xminorticks=true,
xlabel style={font=\bfseries},
xlabel={$c$},
ymode=log,
ymin=1.53111489755082e-06,
ymax=2.56058622395426,
yminorticks=true,
xtick={1e0, 1e1, 1e2, 1e3},
ytick={1e0, 1e-1, 1e-3, 1e-5, 1e-6},
ylabel style={font=\normalsize},
ylabel={$\text{Error}$},
axis background/.style={fill=white},
legend style={yshift=0.0cm,xshift=0.0cm,nodes={scale=0.9, transform shape},legend cell align=left,align=left,draw=white!10!black,legend pos=south west}
]
\addplot [color=mycolor1,solid,line width=2.0pt,mark=o,mark options={solid}]
  table[row sep=crcr]{%
1	2.49451854245962\\
2	2.56058622395426\\
2.82842712474619	2.41549477587747\\
4	2.31910241085377\\
5.65685424949238	2.3390646659369\\
8	2.3559345097313\\
11.3137084989848	2.36317077961077\\
16	1.78602114943219\\
22.6274169979695	1.05664806831337\\
32	0.570405579020334\\
45.2548339959391	0.29560406213711\\
64.0000000000001	0.150266725945533\\
90.5096679918782	0.0757217086700728\\
128	0.0379865777363013\\
181.019335983756	0.0190324273859306\\
256	0.00952528670150864\\
362.038671967513	0.00476537414839776\\
512.000000000001	0.00238406272750457\\
724.077343935026	0.00119304989484375\\
1024	0.000597525223731854\\
};
\addlegendentry{$\Vert z_{\text{UA}} - z_{\text{lim}} \Vert_{H^1}$};

\addplot [color=mycolor2,solid,line width=2.0pt,mark=triangle,mark options={solid}]
  table[row sep=crcr]{%
1	0.0789268359820593\\
2	0.0671259668364969\\
2.82842712474619	0.044214135802572\\
4	0.0291554588233064\\
5.65685424949238	0.0177562695026546\\
8	0.0120915985769484\\
11.3137084989848	0.00894891663251743\\
16	0.00553173146706998\\
22.6274169979695	0.00296902453915092\\
32	0.00151702243379059\\
45.2548339959391	0.00076371450957011\\
64.0000000000001	0.000382729523380365\\
90.5096679918782	0.000191533286950008\\
128	9.58132650754035e-05\\
181.019335983756	4.79299224974959e-05\\
256	2.39832349865012e-05\\
362.038671967513	1.200873158773e-05\\
512.000000000001	6.02116118082157e-06\\
724.077343935026	3.02732978422887e-06\\
1024	1.53111489755082e-06\\
};
\addlegendentry{$\Vert n_{\text{UA}} - n_{\text{lim}} \Vert_{L^2}$};

\addplot [color=mycolor3,dashed,line width=4.0pt]
  table[row sep=crcr]{%
1	1\\
2	0.25\\
2.82842712474619	0.125\\
4	0.0625\\
5.65685424949238	0.03125\\
8	0.015625\\
11.3137084989848	0.00781249999999999\\
16	0.00390625\\
22.6274169979695	0.001953125\\
32	0.000976562499999999\\
45.2548339959391	0.000488281249999999\\
64.0000000000001	0.000244140625\\
90.5096679918782	0.0001220703125\\
128	6.10351562499999e-05\\
181.019335983756	3.05175781249999e-05\\
256	1.52587890625e-05\\
362.038671967513	7.62939453124998e-06\\
512.000000000001	3.81469726562499e-06\\
724.077343935026	1.90734863281249e-06\\
1024	9.53674316406247e-07\\
};
\addlegendentry{$\mathcal{O}(c^{-2})$};

\end{axis}
\end{tikzpicture}%
}
	\caption{Asymptotic convergence plot (double logarithmic scale): Error versus $c$. The oscillatory integrator \eqref{scheme} converges  asymptotically with rate $c^{-2}$ to the numerical solution of the corresponding Zakharov  limit system \eqref{zaki}.}
	\label{fig3}
\end{figure}
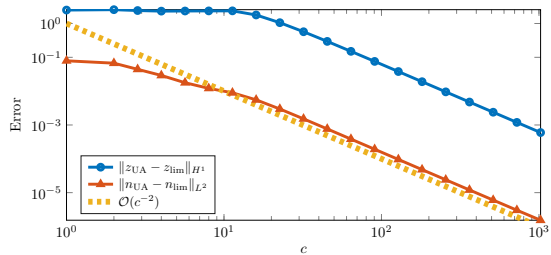
\Red 		
		\begin{rem}
The numerical experiments confirm our theoretical convergence results and in particular stress the uniform convergence property of the novel first-order \eqref{scheme} and second-order \eqref{scheme2ord} oscillatory integrators with respect to the high-plasma frequency $c$ (cf. Theorem \ref{thm:conv} and Theorem \ref{thm:conv2}, respectively) as well as  asymptotic convergence  (cf. \eqref{b:asymp}) towards the Zakharov limit system. The numerical experiments suggest that the error constant of the second-order scheme \eqref{scheme2ord} is even smaller for large values of $c$. In nonstiff regimes $c = 1$ energy conservative schemes for the KGZ system \eqref{kgz} have been naturally proposed in  literature (see, e.g., \cite{WCZ} for a classical finite difference discretization and \cite[Section 2.2.]{BaoKGZ1} for a  finite difference integrator sine pseudospectral method). The new class of oscillatory integrators is in contrast based on an exponential-type discretization which is not symplectic (see also \cite{BFS17} for the classical Klein-Gordon setting $n = - \vert z\vert^2$). Numerical schemes which are both energy conserving \emph{and} at the same time uniformly accurate pose an interesting open research question.
\end{rem}\black

\end{document}